\pdfoutput=1 
\newif\ifpreprint
\preprinttrue 
\ifdefined\ispreprint
  \preprinttrue
\fi
\ifdefined\isjournal
  \preprintfalse
\fi
\ifpreprint

  \documentclass[a4paper,12pt]{article}
  \usepackage{amsthm}
  \usepackage[a4paper]{geometry}
  \usepackage[T1]{fontenc}
  \usepackage[hidelinks,hypertexnames=false,hyperindex=true,pdfpagelabels,linktoc=all]{hyperref}
  \newtheorem{theorem}{Theorem}
  \newtheorem{lemma}[theorem]{Lemma}
  \newtheorem{definition}{Definition}
  \newcommand{\email}[1]{\protect\href{mailto:#1}{#1}}

\else

  \documentclass[review]{siamart171218} 

\fi

\usepackage[utf8]{inputenc}
\usepackage{microtype}

\usepackage[normalem]{ulem} 

\usepackage{amsmath}
\usepackage{amsfonts,amssymb}
\usepackage{mathdots}

\usepackage{algorithm}
\usepackage{algpseudocode}

\usepackage{graphicx}
\graphicspath{{}{tikz/}}
\usepackage{microtype}
\usepackage{pgfplots}
\usepgfplotslibrary{groupplots}
\usetikzlibrary{matrix}
\pgfplotsset{compat=newest} %
\usepgfplotslibrary{external} 
\pgfplotsset{table/search path={data},}

\newcommand{\imagunit}{\mathrm{i}}
\newcommand{\RNum}[1]{\uppercase\expandafter{\romannumeral #1\relax}}
\newcommand{\bszero}{\boldsymbol{0}} 
\newcommand{\bsh}{{\boldsymbol{h}}}    
\newcommand{\bsk}{{\boldsymbol{k}}}    
\newcommand{\bsg}{\boldsymbol{g}} 
\newcommand{\bstau}{\boldsymbol{\tau}}    
\newcommand{\bsp}{\boldsymbol{p}}    
 
\newcommand{\bsn}{{\boldsymbol{n}}} 
\newcommand{\bsell}{\boldsymbol{\ell}}    
\newcommand{\bsx}{\boldsymbol{x}}    
\newcommand{\bsX}{\boldsymbol{X}}    
\newcommand{\bsu}{\boldsymbol{u}}
\newcommand{\bsy}{\boldsymbol{y}}    
\newcommand{\bsz}{\boldsymbol{z}}    
\newcommand{\bsZ}{\boldsymbol{Z}}    
\newcommand{\bsxi}{\boldsymbol{\xi}}
\newcommand{\rd}{\,\mathrm{d}}
\newcommand{\Z}{\mathbb{Z}}
\newcommand{\N}{\mathbb{N}}
\newcommand{\R}{\mathbb{R}}
\newcommand{\Q}{\mathbb{Q}}
\newcommand{\C}{\mathbb{C}}
\newcommand{\calA}{\mathcal{A}}

\newcommand{\tpmod}[1]{~(\operatorname{mod}{#1})} 
\newcommand*{\0}{\mathcal{O}} 
\newcommand{\twopii}{2\pi \imagunit\,}
\newcommand*{\T}{\mathbb{T}}

\newcommand{\rme}{\mathrm{e}}

\newcommand{\tsc}{\gamma} 

\DeclareMathOperator{\diag}{diag}

\newcommand{\Deltat}{{\Delta t}}
\newcommand{\Deltatp}{{(\Delta t)}}

\definecolor{darkred}{RGB}{139,0,0}
\definecolor{darkgreen}{RGB}{0,130,70}
\definecolor{darkmagenta}{RGB}{139,0,139}
\definecolor{darkorange}{RGB}{180,60,0}

\allowdisplaybreaks

\begin{document}

\title{Strang splitting in combination with rank-$1$ \\
and rank-$r$ lattices for the time-dependent Schrödinger equation}
\author{Yuya Suzuki\thanks{KU Leuven, Belgium (\email{yuya.suzuki@cs.kuleuven.be}, \email{dirk.nuyens@cs.kuleuven.be}).}
\and Gowri Suryanarayana\thanks{EnergyVille and VITO, Belgium (\email{gowri.suryanarayana@vito.be}).}
\and Dirk Nuyens\footnotemark[1]}

\maketitle

\begin{abstract}
We approximate the solution for the time dependent Schrödinger equation (TDSE) in two steps. We first use a pseudo-spectral collocation method that uses samples of the functions on rank-$1$ or rank-$r$ lattice points with unitary Fourier transforms.
We then get a system of ordinary differential equations in time, which we solve approximately by stepping in time using the Strang splitting method.
We prove that the numerical scheme proposed converges quadratically with respect to the time step size, given that the potential is in a Korobov space with the smoothness parameter greater than $9/2$.
Particularly, we prove that the required degree of smoothness is independent of the dimension of the problem. We demonstrate our new method by comparing with results using sparse grids from \cite{G07}, with several numerical examples showing large advantage for our new method and pushing the examples to higher dimensionality.
The proposed method has two distinctive features from a numerical perspective: (i) numerical results show the error convergence of time discretization is consistent even for higher-dimensional problems; (ii) by using the rank-$1$ lattice points, the solution can be efficiently computed (and further time stepped) using only $1$-dimensional Fast Fourier Transforms.
\end{abstract}

\section{Introduction}\label{sec:1}

Approximating the solution of the many-particle Schrödinger equation is a challenging problem,
where the dimension of the problem increases linearly with the number of particles in the system.
Many attempts have been made to break the curse of dimensionality with respect to this problem \cite{G07_1, G07,jahnke2000error}. This is also the focus of the present paper and we propose a numerical method which provides a partial solution to this.
Often in the context of physics, the time-dependent Schr\"{o}dinger equation (TDSE) is referred to as the following equation:
\begin{align*}
  \imagunit \, \hbar \, \frac{\partial \psi}{\partial t}(\bsx,t)
  &=
  -\frac{\hbar^2}{2m} \, \nabla^2 \psi(\bsx, t) + v(\bsx) \, \psi(\bsx, t),
\end{align*}%
where $\hbar$ is the reduced Planck constant and $m$ is the mass.
By scaling the time by $1/\sqrt{m}$ and setting $\tsc=\hbar/\sqrt{m}$ this is equivalent to the following form for which $\psi(\bsx,t) = u(\bsx,t/\sqrt{m})$.
We therefore consider the following equivalent equation in this paper (as was done in \cite{G07_1, G07,jahnke2000error}):
\begin{align}\label{TDSE} 
    \imagunit \, \tsc \, \frac{\partial u}{\partial t}(\bsx, t)
    &=
    -\frac{ \tsc^2}{2} \, \nabla^2 u(\bsx, t) + v{(\bsx)} \, u(\bsx, t)
    ,
\end{align}
with positions $\bsx \in \T^d = \T([0,1)^d)$, time $t \in [0,T]$, $\tsc = \hbar / \sqrt{m} > 0$ a small positive parameter, $\imagunit$ the imaginary unit and $\nabla^2$ is the Laplace operator w.r.t.\ the positions $\bsx$, i.e., 
$\nabla^2 = \sum_{i=1}^M \sum_{j=1}^D \partial^2 / \partial x_{i,j}^2$ where $M$ is the number of particles and $D$ is the physical dimensionality.
For notational simplicity we set $d = M \times D$.
The function $u(\bsx,t)$ is the wave function which we seek to approximate, $v(\bsx)$ the potential and $g(\bsx)$ the initial condition at time $t=0$; specific details about these functions will be covered in the later sections. In addition, the boundary conditions are assumed to be periodic. This periodic boundary makes the problem equivalent to identify the domain of $\bsx$ as the $d$-dimensional torus $\T^d = \T([0,1)^d) \simeq [0,1)^d$ with period~$1$.
The TDSE in the above form appears in quantum mechanics and molecular chemistry, and is general enough to include the case of the quantum-mechanical harmonic oscillator, see, e.g., \cite{bandrauk1992higher,Y2010}. We note that this form of equations can be interpreted in several ways: one-particle in $d$-dimensional space; multiple $d$ particles in one-dimensional space (e.g., \cite{MR0280103}); and the combination of those two (multiple particles in multi-dimensional space, e.g., \cite{MR3100763,Y2010}).

In \cite{jahnke2000error}, Jahnke and Lubich applied the Strang splitting method which is an operator splitting method, to approximate the solution of the TDSE where a collocation method using regular grids was first used to discretize the spatial dimensions of the initial wave function and the Strang splitting method was then applied to propagate the wave function in time.
In \cite{G07_1, G07}, sparse grids were used instead of regular grids to overcome the curse of dimensionality but with limited success. The numerical experiments on the TDSE were limited to dimension~$5$.

We are interested in using rank-$1$ lattices for function approximation. Lattice rules have traditionally been used for numerical integration of periodic functions, see, e.g., \cite{CN2008,Nuy2014,SJ1994}. 
Rank-$1$ lattice rules have been studied for the integration of functions belonging to smooth permutation invariant function spaces in \cite{NSW14}. This research is also relevant to our work, since a system with identical particles admits to the setting where (groups of) coordinates, i.e., per particle, are permutation invariant.
Additionally, lattice rules have been used for function approximation in recent years, e.g., \cite{KSW2006,KWW09}. A spectral collocation method using a rank-$1$ lattice was developed by \cite{fred} to approximate the solution of partial differential equations in a periodic space.
In addition to the periodic setting, rank-$1$ lattices, after an appropriate transformation, were found suitable for integration and approximation of non periodic functions from smooth half-period cosine spaces (which includes the usual Sobolev space with bounded mixed first derivatives), see respectively \cite{DNP2012} and \cite{cools2016tent,SNC2015}.

The above research motivates the use of rank-$1$ lattices for solving the TDSE where some symmetry is exhibited due to the physical nature, see \cite{Y2010}.
We derive a spectral collocation method based on rank-$1$ and more general rank-$r$ lattice rules. The general rank-$r$ lattice points also include the (possibly anisotropic) regular grids. The computation of the involved spectral coefficients can be efficiently calculated using unitary Fast Fourier Transformations (FFTs) owing to the special structure of lattice points. Further, we conduct the error analysis of the numerical scheme. Our focus is on the error coming from the time discretization.
The main theoretical result is that the error of the time-discretization converges with rate of $\0(\Deltatp^2)$ where $\Deltat$ is the discretization step of the time $t \in [0,T]$. Our analysis shows that the convergence rate requires some smoothness of the potential function $v(\bsx)$, but this smoothness does not depend on the dimension~$d$, where the results in \cite{G07}, where the collocation was done using sparse grids, need the smoothness to be higher when $d$ increases.
We provide numerical results in various settings, showing that the convergence rate against the time propagation is very stable and not affected by the dimension~$d$. 

The rest of this paper is organized as follows: Section~\ref{sec:2} describes our method and the corresponding theoretical results.
In Section~\ref{sec:3}, we demonstrate our method by numerical experiments with various values of parameters. Our numerical experiments contain both low-dimensional cases and high-dimensional cases.
In Section~\ref{sec:4}, an expression for the total error bound of the full discretization is given.
Section~\ref{sec:5} gives conclusions of this paper.

Throughout the paper $\Z$ denotes the set of all integers, $\Z_n := \{0,1,\dots,n-1 \}$ is the set of integers modulo $n$, $\N := \{1,2,\ldots\}$ the natural numbers and $\Q$ the rational numbers.
We use $I_n$ or just $I$ to denote the $n \times n$ identity matrix.

\section{The method}\label{sec:2}

In this section, we will describe the numerical method used for solving the TDSE. First, we introduce the key concepts that are required through out this paper: lattice point sets, the Fourier pseudo-spectral method on lattices, and the Strang splitting.

\subsection{Lattices}

The main building blocks of the proposed method are \emph{integration lattices}.
They are the intersection of a lattice $A\,\Z^d$ with the unit cube $[0,1)^d$ where $A \in \Q^{d\times r}$, $1 \le r \le d$, is a rational matrix and were originally proposed to approximate periodic integrals on $[0,1)^d$. For more detailed information we refer to \cite{CN2008,lyness1989introduction,SJ1994}.

For the main part of this paper we make use of a \emph{rank-$1$} lattice
\[
  \Lambda(\bsz, n)
  :=
  \left\{ \frac{\bsz k}{n} \bmod 1 \mathrel{:} k \in \Z \right\}
  ,
\]
which is completely defined by its integer \emph{generating vector} $\bsz \in \Z^d$ and the modulus $n$. We take the components of $\bsz$ relatively prime to $n$, such that the total number of points is $n$.

We derive theory for both rank-$1$ and rank-$r$ lattices, enabling us to state
all results for regular (anisotropic) grids  as well since they can be represented by rank-$r$ lattices.
We therefore introduce the definition of a rank-$r$ lattice, see \cite{SJ1994}.

\begin{definition}[Canonical form of rank-$r$ lattice]\label{def:rank-r-lattice}
A $d$-dimensional integration lattice can be written in terms of a generator
\[
  A
  =
  \begin{pmatrix}
    &  &  & \\
    \bsz_1/n_1 & \bsz_2/n_2 & \cdots & \bsz_r/n_r \\
    & & &
  \end{pmatrix}
  \in \Q^{d \times r}
  ,
\]
which is specified by the generating vectors $\bsZ = ( \bsz_1 , \ldots, \bsz_r ) \in \Z^{d \times r}$ and moduli $\bsn = (n_1, \ldots, n_r) \in \N^r$, such that $A =\bsZ \, \diag(\bsn)^{-1}$, with the corresponding lattice point set $\Lambda(\bsZ, \bsn)$ given by
\[
  \Lambda(\bsZ ,\bsn)
  :=
  \Bigl\{  A \bsk \bmod 1 \mathrel{:} \bsk \in \Z^r \Bigr\}
  \; \subset [0,1)^d
  .
\]
This form is the canonical form of a rank-$r$ lattice provided
the moduli satisfy $n_{i+1}$ divides $n_{i}$ for $i=1,\ldots,r-1$,
the generating vectors $\bsz_1, \ldots, \bsz_r \in \Z^d$ are linearly independent over the rational numbers and the components of each $\bsz_i$ are relatively prime to $n_i$.
Then $r$ is the minimum number of generating vectors needed to describe this lattice point set and its total number of unique points in the unit cube $[0,1)^d$ is $n = \prod_{i=1}^r n_i$.
\end{definition}

For further details we refer to \cite[Theorem~3.2]{SJ1994} and the related part there.
We interpret the collection of generating vectors $\bsZ = (\bsz_1,\ldots,\bsz_r) \in \Z^{d\times r}$ as a matrix where the generating vectors constitute the columns of the matrix.
The associated rank-$r$ ``lattice rule'' is the equal-weight cubature rule to approximate an integral of a function $f$ over the unit cube.
For a rank-$r$ lattice in its canonical form we can iterate over all points by a multiindex $\bsk \in \Z_{n_1} \oplus \cdots \oplus \Z_{n_r}$ and therefore the cubature rule based on this lattice point set can be written as
\[
  Q(f; \bsZ, \bsn)
  :=
  \frac1{n_1} \sum_{k_1=0}^{n_1-1} \cdots \frac1{n_r} \sum_{k_r=0}^{n_r-1}
    f\left(\left(\frac{\bsz_1 k_1}{n_1} + \cdots + \frac{\bsz_r k_r}{n_r}\right)\bmod{1}\right)
  .
\]
In this paper we will always assume that a rank-$r$ lattice is given in canonical form, i.e., $\bsZ$ and $\bsn$ satisfy the properties of Definition~\ref{def:rank-r-lattice},
and there is thus a one-to-one correspondence between the lattice points and the multiindex $\bsk \in \Z_{n_1} \oplus \cdots \oplus \Z_{n_r}$.
We also introduce an associated anti-aliasing index set for the rank-$r$ lattice $\Lambda(\bsZ,\bsn)$ which we will denote by $\calA(\bsZ,\bsn)$.
The anti-aliasing set is not unique.

\begin{definition}[Anti-aliasing set]\label{def:rank-r-antialiasing}
  An anti-aliasing set $\calA(\bsZ,\bsn) \in \Z^d$ associated with the rank-$r$ lattice $\Lambda(\bsZ,\bsn)$ in canonical form has the property that for all distinct vectors $\bsh, \bsh' \in \calA(\bsZ,\bsn)$ it never holds that
\[
  \bsZ^{\top} (\bsh-\bsh')
  \equiv
  \bszero
  \pmod{\bsn}
  \qquad
  \equiv
  \qquad
  \begin{cases}
     0 \pmod{n_1}, \\
     \vdots \hspace{3em} \vdots \\
     0 \pmod{n_r},
  \end{cases}
\]
where the equivalence is to be interpreted component-wise and $\bszero$ is the $r$-dimensional zero-vector.
\end{definition}

The anti-aliasing condition states that every $\bsh \in \calA(\bsZ,\bsn)$ can be associated with a unique multiindex $\bsxi \in \Z_{n_1} \oplus \cdots \oplus \Z_{n_r}$, similarly as how we iterate over the points of the rank-$r$ lattice.
Therefore the maximum size of $\calA(\bsZ,\bsn)$ is $n = \prod_{i=1}^r n_i$.
Furthermore, if $|\calA(\bsZ,\bsn)|=n$ we can divide $\Z^d$ into conjugacy classes with respect to $\calA(\bsZ,\bsn)$ in the following three ways
\newlength\sumd
\settowidth{\sumd}{$\scriptstyle \bsxi\in\Z_{n_1}\oplus\cdots\oplus\Z_{n_r}$}
\begin{align}
\begin{split}
  \Z^d
  &=
  \biguplus_{\makebox[\sumd]{$\scriptstyle \bsh \in \Lambda^\perp(\bsZ,\bsn)$}} \left( \bsh + \calA(\bsZ,\bsn) \right)
  \\
  &=
  \biguplus_{\makebox[\sumd]{$\scriptstyle \bsh \in \calA(\bsZ,\bsn)$}} \{ \bsh' \in \Z^d : \bsZ^\top\,\bsh' \equiv \bsZ^\top\,\bsh \pmod{\bsn} \}
  \\
  &=
  \biguplus_{\makebox[\sumd]{$\scriptstyle \bsxi\in\Z_{n_1}\oplus\cdots\oplus\Z_{n_r}$}} \{ \bsh \in \Z^d : \bsZ^\top\,\bsh \equiv \bsxi \pmod{\bsn} \}
  ,
  \end{split}
  \label{eq:disjunct}
\end{align}
where $\uplus$ means that all sets are disjunct.

This set in the rank-$1$ case, $\calA(\bsz,n)$, has been studied before, e.g., \cite{cools1996minimal,fred,cools2016tent,SNC2015}. It is sometimes also called a reconstructing rank-$1$ lattice, e.g., in \cite{byrenheid2016tight,kammerer2014reconstructing}.
By using the concept of the \emph{dual} of the lattice, defined by
\begin{align}
\label{eq:dual}
  \Lambda^\bot(\bsZ,\bsn)
  :=
  \{ \bsh \in \Z^d \mathrel{:} \bsZ^{\top}\,\bsh \equiv \bszero \pmod \bsn \}
  ,
\end{align}
where again the equivalence is to be interpreted component-wise with respect to the $\bsn = (n_1,\ldots,n_r)$,
an anti-aliasing set can be equivalently defined as a set for which for all distinct $\bsh,\bsh'\in \calA(\bsZ,\bsn)$ we have
\begin{align*}
  \bsh-\bsh' \notin \Lambda^\bot(\bsZ,\bsn)
  .
\end{align*}

In addition, we make extensive use of the \emph{character property} of a rank-$r$ lattice
which is given by
\begin{align}\label{char}
  \frac{1}{n}\sum_{\bsp\in\Lambda(\bsZ,\bsn)}\exp(\twopii (\bsh\cdot\bsp))
  =
  \begin{cases}
    1 & \text{if} \; \bsh\in\Lambda^\bot(\bsZ,\bsn), \\
    0 & \text{otherwise}.
  \end{cases}
\end{align}
This follows since
\[
\frac{1}{n}\sum_{\bsp\in\Lambda(\bsZ,\bsn)}\exp(\twopii (\bsh\cdot\bsp))
  =
\prod_{j=1}^r \frac{1}{n_j} \sum_{k_j\in\Z_{n_j}} \exp\left(\twopii (\bsh^\top \bsz_j)\frac{k_j}{n_j}\right),
\]
and
\[
\frac{1}{n_j} \sum_{k_j\in\Z_{n_j}} \exp\left(\twopii (\bsh^\top \bsz_j)\frac{k_j}{n_j}\right)
  =
  \begin{cases}
    1 & \text{if} \; \bsh^\top \bsz_j \equiv 0 \pmod{n_j}, \\
    0 & \text{otherwise}.
  \end{cases}
\]
The condition $\bsh^\top \bsz_j \equiv 0 \pmod{n_j}$ for all $j=1,...,r$ is equivalent to $\bsh \in \Lambda^\bot(\bsZ,\bsn)$, see \eqref{eq:dual}.

As an example of rank-$1$ lattice points and corresponding anti-aliasing set, we exhibit the case for $n=55$ and $\bsz^\top=(1,34)$ in Fig.~\ref{fig:latticeAset}.
Typically, psuedo-spectral Fourier methods use regular grids in the spatial domain, and the unitary Fourier transform maps these points to integer points in a hyper-rectangle in the frequency domain.
Our method uses rank-$r$ lattices instead  of regular grids in the spatial domain.
Similar to the typical pseudo-spectral Fourier methods, a unitary discrete Fourier transformation maps the lattice points to anti-aliasing integer point sets in the frequency domain, see Theorem~\ref{thm:properties}.

 \begin{figure} \label{fig:latticeAset}
 \centering
\includegraphics{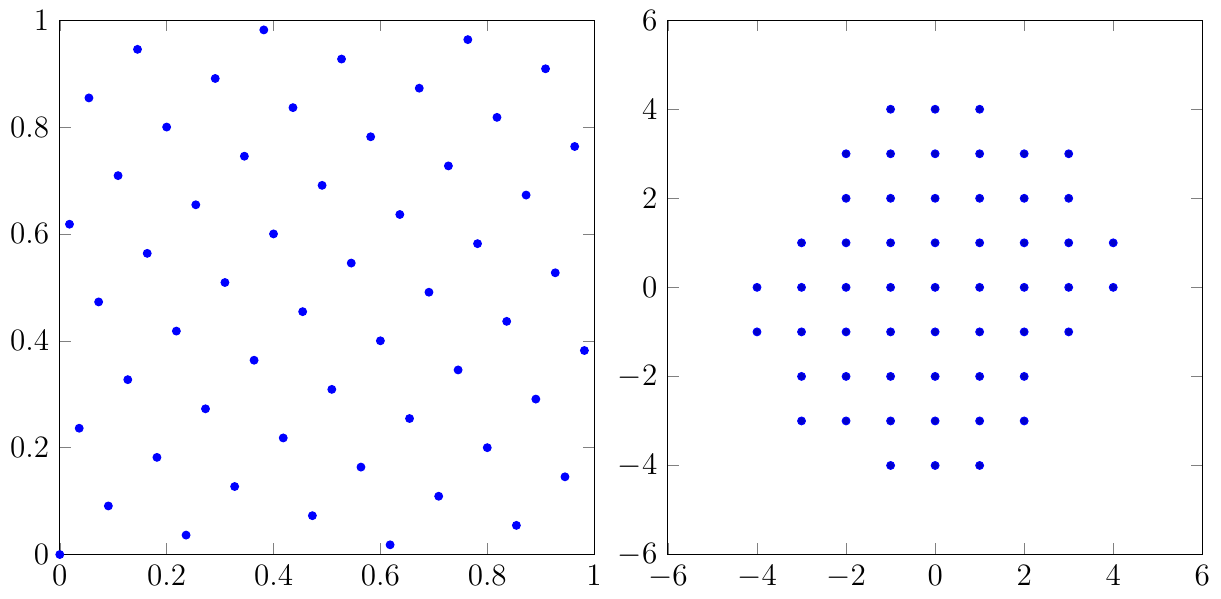}
 \caption{An example of rank-$1$ lattice and the corresponding anti-aliasing set with full cardinality, where $n=55$ and $\bsz^\top=(1,34)$.}
\end{figure}

\subsection{The Fourier pseudo-spectral method on lattice point sets}

A \emph{pseudo-spectral method} is a way to approximate solutions of partial differential equations in terms of a finite number of basis functions.
This was applied to approximate the solution of the TDSE in \cite{G07,jahnke2000error} by expanding all functions into Fourier series.
To apply the Fourier pseudo-spectral method, we require some properties.
The minimum requirement we need is that any considered function is continuous and its Fourier series converges point wise to the original function
\[
  f(\bsx)
  =
  \sum_{\bsh\in\Z^d} \widehat{f}(\bsh) \, \exp(\twopii \bsh\cdot \bsx),
  \qquad
  \text{for all } \bsx \in \T^d
  ,
\]
where the Fourier coefficients of $f$ are given by $\widehat{f}(\bsh)=\int_{[0,1]^d}f(\bsx)\exp(-\twopii\bsh\cdot\bsx)\rd \bsx$.
This condition is guaranteed if we assume that the Fourier coefficients of the function $f$ are absolutely summable
\begin{equation*}
  \|f\|_{A(\T^d)}
  :=
  \sum_{\bsh\in \Z^d} |\widehat{f}(\bsh)| < \infty
  .
\end{equation*}
The space of functions satisfying this condition is called the \emph{Wiener algebra} $A(\T^d)$.
For a detailed discussion, we refer to \cite[Section~3.3]{MR3243734}.

To assure that the solution $u(\bsx,t)\in A(\T^d)$, we have the following lemma
which makes use of the \emph{Korobov space}, a reproducing kernel Hilbert space of Fourier series with a certain decay
\[
  E_\alpha(\T^d)
  := 
  \left\{ f \in L_2(\T^d) \mathrel{:} 
    \|f\|^2_{E_\alpha(\T^d)}
    :=
    \sum_{\bsh\in \Z^d} |\widehat{f}(\bsh)|^2 \, r^2_{\alpha}(\bsh) 
    <
    \infty
  \right\},
\]
where
\begin{align}\label{eq:ralpha}
  r^2_{\alpha}(\bsh)
  :=
  \prod_{j=1}^d \max(|h_j|^{2\alpha},1)
  .
\end{align}
The parameter $\alpha > 1/2$, the smoothness parameter, determines the rate of decay of the Fourier coefficients. 
For $\alpha > 1/2$ we have $E_\alpha(\T^d) \subset A(\T^d)$.
This space is also referred to as a kind of unanchored periodic Sobolev space with dominating mixed-smoothness.
In particular when $\alpha \in \N$ the norm can be expressed in terms of derivatives.
Furthermore, if we define $(r^*_\alpha(\bsh))^2 := \prod_{j=1}^d (1 + |2\pi\,h_j|^{2\alpha}) \ge r^2_\alpha(\bsh)$, then the associated norm defined as above is always larger than for $r_\alpha$, and when $\alpha \in \N$ this norm then reads as
\[
  \|f\|^2_{E^*_\alpha(\T^d)}
  :=
  \sum_{\bsh\in \Z^d} |\widehat{f}(\bsh)|^2 \, (r^*_{\alpha}(\bsh))^2
  =
  \sum_{\bstau \in \{0,\alpha\}^d} \|D^{\bstau} f\|^2_{L_2(\T^d)}
  .
\]
In fact this could be used as an alternative norm throughout the paper.
For a detailed discussion about Korobov spaces, see \cite{NW08} and references therein.
To assure that the term $\nabla^2 u$ in~\eqref{TDSE} makes sense, we require $\alpha \ge 2$.
In the later section, this space plays an important role to prove the convergence of our proposed method.

\begin{lemma}[Regularity of solution and Fourier expansion]\label{lem:u-series}
Given the TDSE~\eqref{TDSE} with $v,g \in E_\alpha(\T^d)$ and $\alpha \ge 2$, then the solution $u(\bsx,t) \in E_\alpha(\T^d)$ for all finite $t \ge 0$ and therefore
\begin{align}\label{eq:u-series}
  u(\bsx,t)
  =
  \sum_{\bsh \in \Z^d} \widehat{u}(\bsh,t) \, \exp(\twopii \bsh \cdot \bsx)
  ,
\end{align}
with
\begin{equation}\label{eq:Fourier-ODE}
  \imagunit\,\tsc\, \widehat{u}' (\bsh,t)
  =
  2\pi^2\tsc^2 \;\|\bsh\|^2_2 \; \widehat{u} (\bsh,t) + \widehat{f}(\bsh,t),
\end{equation}
for all $\bsh\in\Z^d$, with $\widehat{u}'(\bsh,t) = (\partial/\partial t)\,\widehat{u}(\bsh, t)$ and $\widehat{f}(\bsh,t)$ the Fourier coefficients of $f(\bsx,t) := u(\bsx,t) \, v(\bsx)$.
\end{lemma}
\begin{proof}
To prove that $u(\bsx,t) \in E_\alpha(\T^d)$, we first rewrite the TDSE \eqref{TDSE},
\[
    \frac{\partial}{\partial t} u(\bsx, t)
    =
    \imagunit \, \frac{ \tsc}{2} \, \nabla^2 u(\bsx, t) - \frac{\imagunit}{\tsc}v(\bsx) \, u(\bsx, t)
    =
    A u(\bsx, t) +  B u(\bsx, t),
\]
where $Au(\bsx, t)=\imagunit \, \frac{ \tsc}{2} \, \nabla^2u(\bsx, t)$ and $Bu(\bsx, t)=-\frac{\imagunit}{\tsc}v(\bsx)u(\bsx, t)$. We let $(\rme^{At})_{t \ge 0}$ and $(\rme^{Bt})_{t \ge 0}$ denote strongly continuous semigroups generated by $A$ and $B$ respectively.
We note that the solution of \eqref{TDSE}, then, can be written as $u(\bsx,t)  =  \rme^{(A + B) \, t}  \, g(\bsx) $.
Observe that $\rme^{A t} $ is unitary on the Korobov space $E_\alpha(\T^d)$, i.e., for any $t \ge 0$,
\[
 \|\rme^{A t} g(\bsx) \|^2_{E_\alpha(\T^d)}
 =
 \sum_{\bsh\in \Z^d} |\rme^{\frac{\gamma}{2} \imagunit \; \| \bsh \|^2_2 \; t} \; \widehat{g}(\bsh)|^2 \; r^2_{\alpha}(\bsh)
 =
 \sum_{\bsh\in \Z^d} |\widehat{g}(\bsh)|^2  r^2_{\alpha}(\bsh)
 =
 \|g(\bsx) \|^2_{E_\alpha(\T^d)}.
\]
Also we know that the Korobov space is an algebra (see \cite[Appendix~2]{MR2049868}) such that for any $f, g \in E_\alpha(\T^d)$ also their product is in $E_\alpha(\T^d)$,
\begin{align*}
  \|f \, g \|_{E_\alpha(\T^d)}
  \le 
  C_{d,\alpha} \, \|f\|_{E_\alpha(\T^d)} \, \|g\|_{E_\alpha(\T^d)}
  ,
\end{align*}
where the constant $C_{d,\alpha} = 2^{d\alpha} (1 + 2\zeta(2\alpha))^{d/2}$.
Hence this result holds for any $t \ge 0$ and $f(\bsx) = v(\bsx)$ and $g = u(\bsx, t)$.
By using the Lie--Trotter product formula for unbounded self-adjoint operators (see e.g., \cite[Theorem~6.4]{MR2953553}), we obtain the following bound:
\begin{align*}
\bigl\|\rme^{(A+B) t} u(\bsx,0)  \bigr\|_{E_\alpha(\T^d) }
&=
\Bigl\| \lim_{n\to \infty} \Bigl(\rme^{A \frac{t}{n}} \rme^{ B \frac{t}{n}} \Bigr)^n u(\bsx,0) \Bigr\| _{E_\alpha(\T^d)}\\
&\kern-5em \le
\lim_{n\to \infty} \Bigl(\| \rme^{A \frac{t}{n} } \|_{E_\alpha(\T^d)\to E_\alpha(\T^d)} \;\; \| \rme^{B \frac{t}{n}} \|_{E_\alpha(\T^d)\to E_\alpha(\T^d)}  \Bigr)^n  \|u(\bsx,0)\|_{E_\alpha(\T^d) } \\
& \kern-5em \le
\lim_{n\to \infty} \Bigl( \| \rme^{ B \frac{t}{n} } \|_{E_\alpha(\T^d)\to E_\alpha(\T^d)} \Bigr)^n \|u(\bsx,0)\|_{E_\alpha(\T^d) } \\ 
&\kern-5em \le
\lim_{n\to \infty} \Bigl( \rme^{ \|B\|_{E_\alpha(\T^d)\to E_\alpha(\T^d)} \frac{t}{n}  }    \Bigr)^n \|u(\bsx,0)\|_{E_\alpha(\T^d) }  \\
&\kern-5em \le
\rme^{ \frac{1}{\gamma} C_{d,\alpha} \| v \|_{E_\alpha(\T^d)} t}  \|u(\bsx,0)\|_{E_\alpha(\T^d) } ,
\end{align*}
 where $\|V\|_{X \to X} := \sup_{0 \ne u \in X} \|V(u)\|_X / \|u\|_X$ is the induced operator norm, and the above bound is finite for finite time $t$.
Thus we have $u(\bsx,t) \in E_\alpha(\T^d)$ for any finite $t > 0$.
\\
By expanding the left hand side of~\eqref{TDSE}, we have 
 \[
  \imagunit \, \tsc \, \frac{\partial u}{\partial t}
  =
  \imagunit \, \tsc \, \sum_{\bsh \in \Z^d} \widehat{u}'(\bsh,t) \, \exp(\twopii \bsh \cdot \bsx)
  .
\]
By also expanding the right hand side of~\eqref{TDSE}, we obtain
\[
  \imagunit \, \tsc \, \sum_{\bsh \in \Z^d} \widehat{u}'(\bsh,t) \, \exp(\twopii \bsh \cdot \bsx)
  =
  \sum_{\bsh \in \Z^d} \left( 2\pi^2\tsc^2 \;\|\bsh\|^2_2 \; \widehat{u} (\bsh,t) + \widehat{f}(\bsh,t) \right) \, \exp(\twopii \bsh \cdot \bsx).
\]
This holds for all $\bsx\in\T^d$, therefore by comparing each of the coefficients, we obtain
\begin{equation}
  \imagunit\,\tsc\, \widehat{u}' (\bsh,t)
  =
  2\pi^2\tsc^2 \;\|\bsh\|^2_2 \; \widehat{u} (\bsh,t) + \widehat{f}(\bsh,t),
\end{equation}
for all $\bsh\in\Z^d$.
\end{proof}

We approximate the Fourier series~\eqref{eq:u-series} using a rank-$r$ lattice $\Lambda(\bsZ,\bsn)$ and a corresponding well chosen anti-aliasing set $\calA(\bsZ,\bsn)$. Let the approximation of the solution be given by
\begin{align}\label{approx_u}
  u_a(\bsx,t)
  :=
  \sum_{\bsh\in \calA({\bsZ,\bsn})} \widehat{u}_a(\bsh,t) \, \exp(\twopii\bsh\cdot\bsx)
  ,
\end{align}
with the approximated coefficients calculated by a rank-$r$ lattice rule
\begin{align}\label{approx_uhat}
  \widehat{u}_a(\bsh,t)
  :=
  {\frac1n}
  \sum_{\bsp\in \Lambda({\bsZ,\bsn})}
    u(\bsp,t) \, \exp(-\twopii\bsh\cdot \bsp)
  .
\end{align}
We thus write $u_a(\bsx,t)$ and $\widehat{u}_a(\bsh,t)$ to denote the approximations to $u(\bsx,t)$ and $\widehat{u}(\bsh,t)$ respectively.
For notational simplification we fix the time $t$ and omit this argument in the remainder of this section.

We define the unitary discrete Fourier transform (DFT) to map an $r$-dimensional tensor $\bsx \in \C^{n_1\times\cdots\times n_r}$ to a similarly shaped tensor $\bsX \in \C^{n_1\times\cdots\times n_r}$ by the transform
\begin{align}\label{eq:Fnx}
  X_{\xi_1,\ldots,\xi_r}
  &=
  \frac1{\sqrt{n_1}} \sum_{k_1=0}^{n_1-1} \exp(-\twopii k_1 \xi_1 / n_1) \,
  \cdots
  \frac1{\sqrt{n_r}} \sum_{k_r=0}^{n_r-1} \exp(-\twopii k_r \xi_r / n_r) \,
  x_{k_1,\ldots,k_r}
  ,
\end{align}
for $\bsxi \in \Z^{n_1} \oplus \cdots \oplus \Z^{n_r}$ and with the obvious modification for $r=1$.
We define the unitary one-dimensional Fourier matrix and its inverse by
\begin{align}\label{eq:Fn}
  F_n
  &:=
  \frac{1}{\sqrt{n}}
  \Biggl( \exp(-\twopii \, k \xi / n) \Biggr)_{k=0}^{n-1}
  ,
  &
  F_n^{-1}
  &:=
  \frac{1}{\sqrt{n}}
  \Biggl( \exp(\twopii \, k \xi / n) \Biggr)_{\xi=0}^{n-1}
  ,
\end{align}
and the $r$-dimensional Fourier matrix of size $n_1 \times \cdots \times n_r$ as the tensor product $F_\bsn = \otimes_{i=1}^r F_{n_i}$.
We can then write $\bsX = F_\bsn \, \bsx$ for~\eqref{eq:Fnx}, and $\bsx = F_\bsn^{-1} \, \bsX$ when ``vectorizing'' the tensors in lexicographical ordering.
The fast implementation of transforming $\bsx$ into $\bsX$, as well as its inverse, in $\0(n \log n)$, where $n = \prod_{i=1}^r n_i$, is called the fast Fourier transform (FFT) and is well known (although the direction and the normalization vary from implementation to implementation).

In the next theorem we show how to use $r$-dimensional FFTs to map from a rank-$r$ lattice (in space) to a corresponding anti-aliasing set of full cardinality (in the frequency domain), and back.
Note that a regular grid would be represented as a lattice with $r=d$, and in this setting the usage of the $d$-dimensional FFT is well known.
The use of one-dimensional FFTs with a rank-$1$ lattice and a corresponding anti-aliasing set is also known, see, e.g., \cite{MR2479224,fred}.
We extend this for rank-$r$ lattices by using the $r$-dimensional FFT. The following theorem shows three essential properties which make use of the fact that $|\calA(\bsZ,\bsn)| = n$.

\begin{theorem}\label{thm:properties}
Given a rank-$r$ lattice point set $\Lambda(\bsZ,\bsn)$ in canonical form and a corresponding anti-aliasing set $\calA(\bsZ,\bsn)$ with $|\calA(\bsZ,\bsn)| = n$, the following properties hold.
\\[2mm]
(\textit{i}) (Dual character property)
Define the corresponding $d$-dimensional Dirichlet kernel by
\begin{align*}
  D_{\calA(\bsZ,\bsn)}(\bsx)
  :=
  \sum_{\bsh\in\calA(\bsZ,\bsn)} \exp(\twopii \bsh\cdot\bsx). 
\end{align*}
Then for any two lattice points $\bsp,\bsp' \in \Lambda(\bsZ,\bsn)$
\begin{align}
  \frac{1}{n} D_{\calA(\bsZ,\bsn)}(\bsp-\bsp')
  =
  \frac{1}{n}
  \sum_{\bsh\in\calA(\bsZ,\bsn)} \exp(\twopii \bsh\cdot(\bsp-\bsp'))
  =
  \delta_{\bsp,\bsp'},
\label{Diri2}
\end{align}
where $\delta_{\bsp,\bsp'}$ is the Kronecker delta function that is $1$ if $\bsp=\bsp'$ and $0$ otherwise.
\\[2mm]
(\textit{ii}) (Interpolation condition)
If $u_a$ is the approximation of a function $u \in A(\T^d)$ by truncating its Fourier series expansion to the anti-aliasing set $\calA(\bsZ,\bsn)$ and by calculating the coefficients by the rank-$r$ lattice rule, cfr.~\eqref{approx_u} and~\eqref{approx_uhat}, then for any $\bsp \in \Lambda(\bsZ,\bsn)$
\begin{equation}\label{interpolation}
  u_a(\bsp)
  =
  u(\bsp)
  .
\end{equation}
(\textit{iii}) (Mapping through FFT)
Define the $r$-dimensional tensors
\begin{align*}
  \bsu
  &:=
  \big( u(\bsp_{(k_1,\ldots,k_r)}) \big)_{k_1=0,\ldots,n_1-1,\ldots,k_r=0,\ldots,n_r-1}
  ,
  \\
  \bsu_a
  &:=
  \big( u_a(\bsp_{(k_1,\ldots,k_r)}) \big)_{k_1=0,\ldots,n_1-1,\ldots,k_r=0,\ldots,n_r-1}
  ,
  \\
  \widehat{\bsu}_a
  &:=
  \big( \widehat{u}_a(\bsh_{(\xi_1,\ldots,\xi_r)}) \big)_{\xi_1=0,\ldots,n_1-1,\ldots,\xi_r=0,\ldots,n_r-1}
  ,
\end{align*}
with $\bsp_\bsk = (\bsz_1 k_1 / n_1 + \cdots + \bsz_r k_r / n_r) \bmod{1} \in \Lambda(\bsZ,\bsn)$, and where $\bsh_{\bsxi} \in \calA(\bsZ,\bsn)$ is such that $\bsxi = (\bsh\cdot\bsz_1\bmod{n_1},\ldots,\bsh\cdot\bsz_r\bmod{n_r})$.
Then $\bsu = \bsu_a$ (by (ii)) is the collection of function values $u(\bsp)$ on the lattice points $\bsp \in \Lambda(\bsZ,\bsn)$ and $\widehat{\bsu}_a$ is the collection of Fourier coefficients $\widehat{u}_a(\bsh)$ (by using the lattice rule, cfr.~\eqref{approx_u} and~\eqref{approx_uhat}) on the anti-aliasing indices $\bsh \in \calA(\bsZ,\bsn)$.
The $r$-dimensional discrete Fourier transform and its inverse now maps tensors $\bsu_a \in \C^{n_1\times\cdots\times n_r}$ to tensors $\widehat{\bsu}_a \in \C^{n_1\times\cdots\times n_r}$ and back.
\end{theorem}

\begin{proof}
\mbox{}\\[2mm]
(\textit{i}) 
 The proof is based on \cite[Theorem~7.3]{MR1489255}.
 Remember that $n = \prod_{i=1}^r n_i$.
Now associate an arbitrary but fixed ordering such that we can enumerate the lattice points by $\bsp^{(\kappa)}$ for $\kappa=0,\ldots,n-1$.
Likewise, associate an arbitrary but fixed ordering such that we can enumerate the Fourier indices in the anti-aliasing set by $\bsh^{(\chi)}$ for $\chi = 0,\ldots,n-1$.
Then
\begin{equation}\label{Dirichlet}
  \frac{1}{n}\sum_{\kappa=0}^{n-1} \exp(\twopii \bsh^{(\chi)} \cdot \bsp^{(\kappa)}) \, \exp(-\twopii \bsh^{(\chi')} \cdot \bsp^{(\kappa)}) 
  = 
  \delta_{\chi,\chi'},
  \qquad \text{for all } \chi,\chi'=0,\ldots,n-1
  ,
\end{equation}
because of the character property~\eqref{char} and since $\bsh^{(\chi)}-\bsh^{(\chi')}\notin\Lambda^\bot(\bsZ,\bsn)$ for $\chi\neq\chi'$ because of the anti-aliasing condition.
We rewrite~\eqref{Dirichlet} as
\[
  PMP^{*}
  =
  I_n
  ,
\]
where $M = \diag(1/n,\ldots,1/n)$ and
\begin{align*}
  P
  &=
  {\begin{pmatrix}
    \exp(\twopii \bsh^{(\chi)}\cdot\bsp^{(\kappa)})
  \end{pmatrix}_{\substack{\chi = 0, \ldots, n-1 \\ \kappa = 0, \ldots, n-1 }}}
 \\
  &=
  \begin{pmatrix}
    \exp(\twopii \bsh^{(0)}\cdot\bsp^{(0)})   & \cdots &\exp(\twopii \bsh^{(0)}\cdot\bsp^{(n-1)})  \\
    \vdots & \ddots & \vdots \\
    \exp(\twopii \bsh^{(n-1)}\cdot\bsp^{(0)}) & \cdots& \exp(\twopii \bsh^{(n-1)}\cdot\bsp^{(n-1)})
  \end{pmatrix}
  , 
\end{align*}

with $P^{*}$ the Hermitian conjugate of $P$. 
We note that once~\eqref{Dirichlet} holds, then the matrix $P$ is non-singular. Therefore we obtain
\[
  P^{*} P
  =
  M^{-1}
  ,
\]
which can be written as
\[
  \frac1n
  \sum_{\chi=0}^{n-1} \exp(-\twopii \bsh^{(\chi)} \cdot \bsp^{(\kappa)}) \exp(\twopii \bsh^{(\chi)} \cdot \bsp^{(\kappa')})
  =
  \delta_{\kappa,\kappa'}
  ,
  \qquad \text{for all } \kappa,\kappa'=0,\ldots,n-1
  ,
\]
which is equivalent to~\eqref{Diri2}.
\\[2mm]
(\textit{ii}) From~\eqref{approx_u} and~\eqref{approx_uhat} it follows
\begin{align*}
\begin{split}
  u_a(\bsp) 
  &=
  \sum_{\bsh\in \calA(\bsZ,\bsn)}\widehat{u}_a(\bsh) \, \exp(\twopii\bsh\cdot\bsp) \\
  &=
  \sum_{\bsh\in \calA(\bsZ,\bsn)}\left( \frac{1}{n}\sum_{\bsp'\in \Lambda(\bsZ,\bsn)}u(\bsp') \, \exp(-\twopii \bsh\cdot \bsp')\right) \exp(\twopii\bsh\cdot \bsp) \\
  &=
  \sum_{\bsp'\in \Lambda(\bsZ,\bsn)}u(\bsp') \; \frac{1}{n}\sum_{\bsh\in \calA(\bsZ,\bsn)} \exp(-\twopii \bsh\cdot \bsp') \, \exp(\twopii\bsh\cdot \bsp) \\
  &=
  \sum_{\bsp'\in \Lambda(\bsZ,\bsn)}u(\bsp') \, \delta_{\bsp, \bsp'}\\
  &=
  u(\bsp),
\end{split}
\end{align*}
where the dual character property~\eqref{Diri2} is used for the second to last equality.
\\[2mm]
(\textit{iii})
Consider approximating the Fourier coefficient $\widehat{u}(\bsh)$ by the rank-$r$ lattice rule,
\[
  \widehat{u}_a(\bsh)
  =
  \frac1{n_1} \sum_{k_1=0}^{n_1-1} \cdots \frac1{n_r} \sum_{k_r=0}^{n_r-1}
    u(A\bsk \bmod{1}) \, \exp(-\twopii \bsh^\top A \bsk)
  .
\]
Now define the $r$-dimensional function $v(k_1/n_1,\ldots,k_r/n_r) := u(A\bsk \bmod{1})$, then we can identify the above equation with
\begin{multline*}
  \widehat{u}_a(\bsh)
  =
  \frac1{n_1} \sum_{k_1=0}^{n_1-1} \cdots \frac1{n_r} \sum_{k_r=0}^{n_r-1}
  v(k_1/n_1,\ldots,k_r/n_r) \,
  \prod_{j=1}^r \exp(-\twopii (\bsh\cdot\bsz_j) \, k_j / n_j)
  \\
  =
  \widehat{v}(\bsh\cdot\bsz_1 \bmod{n_1}, \ldots, \bsh\cdot\bsz_r \bmod{n_r})
  ,
\end{multline*}
where $\widehat{v}(\xi_1,\ldots,\xi_r)$ are the discrete Fourier coefficients of $v$.
Now because of the anti-aliasing condition we can identify each $\bsh \in \calA(\bsZ,\bsn)$ uniquely with an index $\bsxi \in \Z_{n_1} \oplus \cdots \oplus \Z_{n_r}$ through $(\bsh\cdot\bsz_1 \bmod{n_1}, \ldots, \bsh\cdot\bsz_r \bmod{n_r}) = (\xi_1, \ldots, \xi_r)$.
Therefore the transformation is an $r$-dimensional $n_1 \times \cdots \times n_r$ discrete Fourier transform.
\end{proof}

Finally we show the relation between the approximated coefficients $ \widehat{u}_a(\bsh)$ and the coefficients $\widehat{u}(\bsh)$.
The approximated coefficients would be exact in case the function $u$ is solely supported on the anti-aliasing set $\calA(\bsZ,\bsn)$, but in general this is not the case and we will have aliasing errors.

\begin{lemma}[Aliasing]\label{lem:aliasing}
  The approximated Fourier coefficients~\eqref{approx_uhat} through the lattice rule $\Lambda(\bsZ,\bsn)$ alias the true Fourier coefficients in the following way
  \[
    \widehat{u}_a(\bsh)
    =
    \sum_{\bsh' \in \Lambda^{\bot}(\bsZ,\bsn)} \widehat{u}(\bsh+\bsh')
    =
    \widehat{u}(\bsh) + \sum_{\bszero \ne \bsh' \in \Lambda^{\bot}(\bsZ,\bsn)} \widehat{u}(\bsh+\bsh')
    .
  \]
\end{lemma}
\begin{proof}
This follows from a straightforward calculation:
\begin{align*}
  \widehat{u}_a(\bsh)
  &=
  \frac{1}{n}\sum_{\bsp\in\Lambda(\bsZ,\bsn)} u(\bsp) \, \exp(-\twopii\bsh\cdot\bsp)
  \\
  &=\frac{1}{n}\sum_{\bsp\in\Lambda(\bsZ,\bsn) }\sum_{\bsh' \in \Z^d} \widehat{u}(\bsh') \,\exp(\twopii \bsh' \cdot \bsp) \, \exp(-\twopii\bsh\cdot\bsp)
  \\
  &=
  \sum_{\bsh' \in \Z^d}\widehat{u}(\bsh') \; \frac{1}{n} \sum_{\bsp\in\Lambda(\bsZ,\bsn)} \exp(\twopii(\bsh'-\bsh)\cdot\bsp)
  \\
  &=
  \sum_{\bsh' \in \Lambda^{\bot}(\bsZ,\bsn)} \widehat{u}(\bsh+\bsh')
  ,
\end{align*} 
where the character property~\eqref{char} is used in the last equality. 
\end{proof}

This last lemma shows that for the approximation $u_a$ to be meaningful the Fourier coefficients necessarily have to decay at a certain rate such that the error in the approximation $\widehat{u}_a(\bsh)$ can be bounded.
This decay is not enforced by asking $u \in A(\T^d)$, but it is enforced by asking $u$ to be in the Korobov space $E_\alpha(\T^d)$.

\subsection{Strang splitting}

We will use Strang splitting to do time stepping on our discretized function.
The idea of Strang splitting \cite{MR0235754} is to break up the solution operator for an ODE which consists of a sum of two differential operators into applying them each separately in a way to be specified below and thereby achieving second order convergence with respect to the time step.
Strang splitting can be applied to initial value problems of the form
\[
  y'(t) = (A + B) \, y(t)
  ,
  \qquad y(0) = y_0,
\]
where $A$ and $B$ are differential operators.

We first explain a splitting method which attains first order convergence in the time step and then show the Strang operator splitting which gives second order convergence in the time step.
If $A$ and $B$ are constant coefficient matrices, as is the case in our application, then the solution is given by
\[
  y(t) = \rme^{(A + B) \, t} \, y_0
  ,
\]
where $\rme^{(A + B)}$ is the matrix exponential.
If $A$ and $B$ commute, i.e., $A B = B A$, then $y(t) = \rme^{A \, t} \, \rme^{B \, t} \, y_0$.
This follows from the Baker--Campbell--Hausdorff formula from Lie group analysis
\[
  \log( \rme^{A\,t} \, \rme^{B\,t} )
  =
  (A + B) \, t + [A,B] \, \frac{t^2}{2} 
  + \left([A,[A,B]]+[B,[B,A]]\right) \, \frac{t^3}{12} 
  + \cdots
  ,
\]
where the commutator of two operators $A$ and $B$ is defined by $[A, B] := A B - B A$,
which reduces to $\rme^{A\,t} \, \rme^{B\,t} = \rme^{(A+B)\,t}$ if $[A,B]=0$, where $0$ should be interpreted as the zero matrix.
If $[A,B]$ is nonzero then, given an initial solution $y(t)$, we can write
\[
  y(t+\Deltat)
  =
  \rme^{(A + B) \, \Deltat} \, y(t)
  =
  \exp\left( \log( \rme^{A\, \Deltat} \, \rme^{B\, \Deltat} ) - [A,B] \, \frac{\Deltatp^2}{2} - \0(\Deltatp^3) \right) \, y(t)
  ,
\]
which for a discrete time stepping scheme $y_t \approx y(t)$ can be used to show a global error of first order in $\Deltat$ for bounded (fixed) commutator.

For Strang operator splitting we first write $A + B = \tfrac12 B + A + \tfrac12 B$, we now want to approximate $\rme^{(A+B)\,t}$ by $\rme^{\frac12 B \, t} \rme^{A \, t} \rme^{\frac12 B \, t}$, in effect taking twice half a time step for $B$ and sandwiching a full time step for $A$ in the middle.
The Strang splitting method for a time discretization $\Deltat$ then operates as follows
\begin{align}\label{eq:Strang-scheme}
  y_{k+1}
  =
  \rme^{\frac12 B \, \Deltat} \, \rme^{A \, \Deltat} \, \rme^{\frac12 B \, \Deltat}
  \, y_k
  ,
\end{align}
where $y_k \approx y(k \Deltat)$ and $y_0 = y(0)$ is the initial value.
We have the following local error bound (per time step) from \cite[Theorem~2.1]{jahnke2000error}.

\begin{theorem}[Strang splitting local error bounds]\label{thm:strang-local-error}
Let $X$ be a Banach space equipped with the norm $\|\cdot\|$, $A$ the generator of the strongly continuous semigroup $\rme^{A\,t}$ on the Banach space $X$, and $B$ be a bounded linear operator on $X$
with induced operator norm $\|B\| := \sup_{0 \ne y \in X} \|B\,y\| / \|y\| < \infty$.
Let $\omega$ be an arbitrary constant.
Then the following hold.
\\[2mm]
(\textit{i}) If there exist constants $\alpha\geq0$ and $c_1$ satisfying
\[
  \|[A,B] \, y\|
  \le
  c_1 \, \|(A+\omega\,I)^{\alpha} \, y\|
  \qquad \text{for all } y\in X
  , 
\]
then the local error of the Strang splitting method is bounded as follows
\[
  \left\|
    \rme^{\frac12 B\,\tau} \, \rme^{A \, \tau} \, \rme^{\frac12 B \, \tau} y
    -
    \rme^{(A+B)\,\tau} \, y
  \right\|
  \le
  C_1 \, \tau^2 \, \|(A+\omega\,I)^{\alpha} \, y \|
  \qquad \text{for all } y\in X
  ,
\]
where $C_1$ depends only on $c_1$ and $\|B\|$.
\\[2mm]
(\textit{ii}) Under the condition above and additionally if there exist constants $\beta \geq 1 \geq \alpha$, and $c_2$ satisfying 
\[
  \|[A,[A,B]]\,y\|
  \le
  c_2 \, \|(A+\omega\,I)^{\beta}y\|
  \qquad \text{for all } y\in X
  ,
\]
then the local error of the Strang splitting method is bounded as follows
\[
  \left\|
    \rme^{\frac12 B\,\tau} \, \rme^{A \, \tau} \, \rme^{\frac12 B \, \tau} y
    -
    \rme^{(A+B)\,\tau} \, y
  \right\|
  \le
  C_2 \, \tau^3 \, \|(A+\omega\,I)^{\beta} \, y\|
  \qquad \text{for all } y\in X
\]
where $C_2$ depends only on $c_1$, $c_2$ and $\|B\|$.
\end{theorem}
\begin{proof}
  See \cite[Theorem~2.1]{jahnke2000error}.
\end{proof}

Now assume $\widehat{\bsu}_t$ are the approximate Fourier coefficients of $u(\bsx, t)$ at time $t$.
The previous theorem shows that we need a bound on $\|(A+\omega\, I)^\alpha \, \widehat{\bsu}_t \|_2$ and that $\|B\|_2$ should be bounded to get first order convergence for the global error of the time stepping scheme using the Strang splitting method.
It also shows that if we have a bound on $\|(A+\omega I)^\beta \, \widehat{\bsu}_t \|_2$ we obtain second order convergence for the global error of the time stepping scheme.

In Lemma~\ref{lem:rank-1-commutator-bounds} we will first derive the key ingredient for our main result when the discretization in space is done by a rank-$1$ lattice $\Lambda(\bsz,n)$ with corresponding finite Fourier series on an associated anti-aliasing set $\calA(\bsz,n)$.
In Lemma~\ref{lem:rank-r-commutator-bounds} we will extend the result for general rank-$r$ lattices $\Lambda(\bsZ,\bsn)$ which include any regular (possibly anisotropic) grid.

\subsection{Strang splitting and rank-$1$ lattices}

Denote by $\widehat{\bsu}_t := \bigl( \widehat{u}_a(\bsh^{(0)},t),$ $\dots,\widehat{u}_a(\bsh^{(n-1)},t) \bigr)$ the approximated solution at time $t$ using a fixed anti-aliasing set $\calA(\bsz,n) = \{ \bsh_\xi \mathrel{:} \xi = 0, \ldots, n-1\}$ of full size $n$, where $\bsh_\xi \in \calA(\bsz,n)$ is such that $\bsh_\xi \cdot \bsz \equiv \xi \pmod{n}$. 
Demanding that~\eqref{eq:Fourier-ODE} holds for all $\bsh\in\calA(\bsz,n)$, we have the following relation
\begin{align}\label{ode}
  \imagunit\,\tsc\, \widehat{\bsu}_t'
  &=
  \frac{1}{2}\tsc^2D_n\widehat{\bsu}_t+W_n\widehat{\bsu}_t,
\end{align}
with the initial condition $\widehat{\bsu}_{0} =\widehat{\bsg}_a:=(\widehat{g}_a(\bsh^{(0)}),\dots,\widehat{g}_a(\bsh^{(n-1)}))$,
\begin{align}\label{eq:Dn}
  D_n
  :=
  \diag\left((4\pi^2\|\bsh_\xi\|_2^2)_{\xi=0,\ldots,n-1}\right)
  ,
\end{align}
and the linear operator $W_n:=F_n V_n F_n^{-1}$ with 
\begin{align}\label{eq:Vn}
  V_n
  :=
  \diag\left( \left( v(\bsp_k) \right)_{k=0,\ldots,n-1} \right)
  ,
\end{align}
where $F_n$ is the unitary Fourier matrix.
For the derivation of $W_n$, we have the following Lemma.
\begin{lemma}[Multiplication operator on rank-$1$ lattices]
Given a rank-$1$ lattice point set $\Lambda(\bsz, n)$ and corresponding anti-aliasing set $\calA(\bsz,n)$ of full size, a potential function $v \in E_\alpha(\T^d)$ with $\alpha \ge 2$ and a function $u_a \in E_\beta(\T^d)$ with $\beta \ge 2$ with Fourier coefficients only supported on $\calA(\bsz,n)$.
Then the action in the Fourier domain restricted to $\calA(\bsz,n)$ of multiplying with $v$, that is $f_a(\bsx) = v(\bsx) \, u_a(\bsx)$, on the nodes of the rank-$1$ lattice, and with $f_a$ having Fourier coefficients restricted to the set $\calA(\bsz,n)$, can be described by a circulant matrix $W_n \in \C^{n\times n}$ with $W_n = F_n \, V_n \, F_n^{-1}$, with $V_n$ given by~\eqref{eq:Vn} and $F_n$ the unitary Fourier matrix~\eqref{eq:Fn}, where the element at position $(\xi,\xi')$ of $W_n$ is given by
\begin{align}\label{eq:Wn-elements}
  w_{\xi,\xi'}
  &
  =
  w_{(\xi-\xi') \bmod{n}}
  =
  \sum_{\substack{\bsh\in\Z^d \\ \bsh\cdot \bsz\equiv \xi-\xi'\tpmod{n}}}\kern-2em \widehat{v}(\bsh)
  .
\end{align}
\end{lemma}%
\begin{proof}
In the following, $F_n$ is the unitary discrete Fourier transformation matrix. We denote the coefficients of the product $v(\bsx)u(\bsx)$ by $\widehat{f}(\bsh)$. For each $\bsh\in\calA(\bsz,n)$ we have
\begin{align*}
 \widehat{f}(\bsh)
 &=
 \int_{[0,1]^d} u_a(\bsx)v(\bsx) \exp(-\twopii \bsh\cdot\bsx )\rd \bsx \\
 &=
 \int_{[0,1]^d} \Biggl( \sum_{\bsh' \in \calA(\bsz,n)} \widehat{u}_a(\bsh') \exp(\twopii \bsh'\cdot\bsx) \Biggr)
 \\
 &\hspace{10em}\times \Biggl( \sum_{\bsh'' \in \Z^d}\widehat{v}(\bsh'') \exp(\twopii \bsh''\cdot\bsx) \Biggr)  \exp(-\twopii \bsh\cdot\bsx )\rd \bsx
 \\
 &= \sum_{\bsh'\in\calA(\bsz,n)} \widehat{v}(\bsh-\bsh') \, \widehat{u}_a(\bsh').
 \end{align*} 
By Lemma~\ref{lem:aliasing} The coefficients calculated on the rank-$1$ lattice points have the form
\[
\widehat{v}_a(\bsh-\bsh')=\sum_{\bsell\in \Lambda^{\bot}(\bsz,n)}\widehat{v}(\bsh-\bsh'+\bsell).
\] 
Therefore we have the following approximation for $\widehat{f}(\bsh)$:
\[
\widehat{f_a} (\bsh)=\sum_{\bsh'\in\calA(\bsz,n)} \sum_{\bsell\in \Lambda^{\bot}(\bsz,n)}\widehat{v}(\bsh-\bsh'+\bsell)\widehat{u}_a(\bsh',t).
\]
We have hence proved the claims of the lemma.
\end{proof}

The exact solution of the ordinary differential equation~\eqref{ode} is
\[
  \widehat{\bsu}_{t}
  =
  \rme^{-\frac{\imagunit}{\tsc} W_n \, t - \frac{\imagunit\tsc}{2} D_n \, t}
  \, \widehat{\bsu}_{0}.
\]
Applying the Strange splitting method~\eqref{eq:Strang-scheme}
then gives us
\begin{equation}
  \widehat{\bsu}_a^{k+1}
  =
  \rme^{-\frac{\imagunit}{2\tsc} W_n \Deltat}\,\rme^{-\frac{\imagunit\tsc}{2} D_n \Deltat}\,\rme^{-\frac{\imagunit}{2\tsc} W_n\Deltat} \, \widehat{\bsu}_a^k
  \qquad \text{for } k=0,1,\dots,m-1,
\label{StrangTDSE}
\end{equation}
where  
\[
  \rme^{-\frac{\imagunit}{2} W_n \Deltat} 
  = 
  F_n \diag\left( (\rme^{-\frac{\imagunit}{2} v(\bsp_k) \Deltat})_{k=0,\ldots,n-1} \right) F_n^{-1}.
\]
We then
approximate the solution of the differential system at time $t=k\Deltat$ by stepping with a time step of $\Deltat$ iteratively.

To bound the error of the Strang splitting we need to bound the effect of the commutators as specified in Theorem~\ref{thm:strang-local-error}, for this we will make use of the Korobov space.
Now we are ready to state our key theoretical result, namely that the Strang splitting has bounded error of the time evolution when the discretization in space is done by a rank-$1$ lattice rule and the truncation of the Fourier series is done on an associated anti-aliasing set.
First we show that the commutators of the operators $D=\frac{\tsc}{2}  D_n$ and $W=\frac{1}{\tsc}W_n$ are bounded in the sense of Theorem~\ref{thm:strang-local-error} with $A=D$, $B=W$ and $\omega=1$.

\begin{lemma}[Rank-$1$ lattice commutator bounds]\label{lem:rank-1-commutator-bounds}
Given a rank-$1$ lattice with generating vector $\bsz \in \Z^d$ and modulus $n$ and a TDSE with a potential function $v \in E_\alpha(\T^d)$ with $\alpha \ge 2$ and an initial condition $g \in E_\beta(\T^d)$ with $\beta \ge 2$.
Let $D = \tfrac{\tsc}2 D_n$ and $W = \frac1\tsc W_n$ with $D_n$ and $W_n = F_n V_n F_n^{-1}$ as defined in~\eqref{eq:Dn} and~\eqref{eq:Wn-elements}, and with $V_n$ as defined in~\eqref{eq:Vn} using the potential function $v$.

If the anti-aliasing set $\calA(\bsz,n) = \{ \bsh_\xi \in \Z^d : \bsh_\xi \cdot \bsz \equiv \xi \pmod{n} \text{ for } \xi = 0,\ldots,n-1 \}$, with full cardinality, is chosen such that it has minimal $\ell_2$ norm, i.e.,
\begin{align}\label{eq:hrepresenter}
  \| \bsh_\xi \|_2
  =
  \min_{\bsh' \in A(\bsz,n,\xi)} \|\bsh'\|_2
  ,
\end{align}
with
\begin{align*}
  A(\bsz,n,\xi)
  :=
  \bigl\{ \bsh \in \Z^d : \bsh \cdot \bsz \equiv \xi \pmod{n} \bigr\}
  ,
\end{align*}
then the following hold.
\\[2mm]
(\textit{i}) If $v \in E_\alpha(\T^d)$ with parameter $\alpha > 5/2$, then for all $\bsy \in \R^n$ we have
\begin{align*}
 &\|[D,W]\,\bsy\|_2 \le c_1 \, \|(D+I)\,\bsy\|_2,
 \end{align*}
 where $c_1$ is a constant independent of $n$ and $\bsy$.
\\[2mm] 
(\textit{ii}) If $v \in E_\alpha(\T^d)$ with $\alpha > 9/2$, then for all $\bsy \in \R^n$ we have
 \begin{align*}
 &\|[D,[D,W]]\,\bsy\|_2 \le c_2 \, \|(D+I)^2\,\bsy\|_2,
\end{align*}
where $c_2$ is a constant, independent of $n$ and $\bsy$. 
\end{lemma}
\begin{proof}
We first prove the first order result (\textit{i}) and then prove the second order result (\textit{ii}).
\\[2mm]
(\textit{i}) Since $(D+I)$ is non-singular, we show 
\[
\|[D,W] \, (D+I)^{-1} \, \bsy\|_2\le c_1 \|\bsy\|_2
\qquad \text{for all } \bsy \in \R^n
.
\]
Hence we need to bound the induced matrix $p$-norm $\| A \|_p := \sup_{\bszero \ne \bsy \in \R^n} \| A \, \bsy \|_p / \| \bsy \|_p$ for $p=2$ for the matrix $A = [D,W] \, (D+I)^{-1} \in \R^{n\times n}$ by an absolute constant.
We have
\begin{align*}
[D,W]
= 
D\,W-W D
=
 \Bigl(
   2\pi^2 \left(\|\bsh_\xi\|^2_2-\|\bsh_{\xi'}\|^2_2\right) \, w_{\xi-\xi'}
 \Bigr)_{\xi,\xi'=0,\ldots,n-1}
 ,
\end{align*}
where the subscript of $w$ should be interpreted modulo~$n$, see~\eqref{eq:Wn-elements}.
For ease of notation we now multiply by $\tsc$ and consider the matrix $M$ defined by
\[
  M
  :=
  \tsc \,
  (D\,W-WD) \, (D+I)^{-1}
  =
  \left(
       \frac{
         \|\bsh_\xi\|^2_2-\|\bsh_{\xi'}\|^2_2
       }{
         \|\bsh_{\xi'}\|^2_2 + c
       }
       \, w_{\xi-\xi'}
     \right)_{\xi,\xi'=0,\ldots,n-1}
  ,
\]
where $c = 1/(2 \pi^2\tsc) > 0$.
Note that $\|[D,W] \, (D+I)^{-1}\|_2=\frac{1}{\tsc}\|M\|_2$.
By H\"{o}lder's inequality, we have $\|M\|_2\le\sqrt{\|M\|_1\| \, M\|_{\infty}}$, therefore we will bound $\|M\|_1 = \max_{\xi'=0,\ldots,n-1} \sum_{\xi=0}^{n-1} |M_{\xi,\xi'}|$ and $\|M\|_{\infty} = \max_{\xi=0,\ldots,n-1} \sum_{\xi'=0}^{n-1} |M_{\xi,\xi'}|$.
Clearly the diagonal elements of $M$ are zero and we can exclude those cases in the following.
For $\|M\|_1$ we obtain
\begin{align*}
 \|M\|_1
 &=
 \max_{\xi'\in\Z_n} \hspace{-1mm}   
 \sum_{\substack{\xi=0\\\xi\ne\xi'}}^{n-1} \left| \frac{\|\bsh_\xi\|^2_2-\|\bsh_{\xi'}\|^2_2}{\|\bsh_{\xi'}\|^2_2 + c} \, w_{\xi-\xi'} \right|
 \le 
 \max_{\xi'\in\Z_n}
  \hspace{-1mm}   
  \sum_{\substack{\xi=0\\\xi\ne\xi'}}^{n-1} \left[
    (1 + \frac{\|\bsh_\xi\|^2_2}{\max(\|\bsh_{\xi'}\|^2_2,c)} ) \, |w_{\xi-\xi'}|    
  \right]
 .
\end{align*}
Note that
\[
  \max_{\xi'\in\Z_n} \sum_{\substack{\xi=0\\\xi\ne\xi'}}^{n-1} |w_{\xi-\xi'}| 
  =
  \sum_{i=1}^{n-1} |w_{i}| 
  \le 
  \sum_{\bsh \in \Z^d} |\widehat{v}(\bsh)|
  =
  \|v\|_{A(\T^d)}
  <
  \infty
  ,
\]
since we assume $v \in A(\T^d)$, see Lemma~\ref{lem:u-series}.
We still need to bound
\begin{align}\label{eq:cont}
  \max_{\xi'\in\Z_n}
  \sum_{\substack{\xi=0\\\xi\ne\xi'}}^{n-1}
    \frac{\|\bsh_\xi\|^2_2}{\max(\|\bsh_{\xi'}\|^2_2,\,c)} \, |w_{\xi-\xi'}|
  =
   \max_{\xi'\in\Z_n} \sum_{\substack{\xi=0\\\xi\ne\xi'}}^{n-1} 
     \frac{\|\bsh_\xi\|^2_2}{\max(\|\bsh_{\xi'}\|^2_2,\,c)}
      \frac{\|\bsh_{\xi-\xi'}\|^2_2}{\|\bsh_{\xi-\xi'}\|^2_2} \, |w_{\xi-\xi'}|
  ,
\end{align}
where also $\bsh_{\xi-\xi'}$ has to be read as $\bsh_{(\xi-\xi') \bmod{n}}$.
Note that $\bsh_{\xi-\xi} = \bsh_0 = \bszero$ is excluded from the sum.
Since the anti-aliasing set is such that $\bsh_\xi$ has minimal $\ell_2$ norm by~\eqref{eq:hrepresenter} we can bound $\|\bsh_\xi\|_2 \le \|\bsh_\xi'\|_2$ for any $\bsh_\xi' \in A(\bsz,n,\xi)$ with the property $\bsh_\xi' \cdot \bsz \equiv \xi \pmod{n}$.
In particular for $\bsh_\xi' = \bsh_{\xi-\xi'} + \bsh_{\xi'}$ since $(\bsh_{\xi-\xi'}+\bsh_{\xi'}) \cdot \bsz \equiv \xi \pmod{n}$ for any choice of $\xi'=0,\ldots,n-1$.
Therefore
\[
  \|\bsh_\xi\|_2^2
  \le
  \|\bsh'_\xi\|_2^2
  =
  \|\bsh_{\xi-\xi'}+\bsh_{\xi'}\|_2^2
  \le
  \|\bsh_{\xi-\xi'}\|^2_2 + 2\,\|\bsh_{\xi-\xi'}\|_2\,\|\bsh_{\xi'}\|_2 + \|\bsh_{\xi'}\|^2_2
  ,
\]
and thus (remembering we have $\bsh_{\xi-\xi'} \ne \bszero$)
\[
  \frac{\|\bsh_\xi\|^2_2}
       {\max(\|\bsh_{\xi'}\|^2_2,\,c) \, \|\bsh_{\xi-\xi'}\|^2_2}
  \le
  \frac{\|\bsh_{\xi-\xi'}\|^2_2 + 2\,\|\bsh_{\xi-\xi'}\|_2\,\|\bsh_{\xi'}\|_2 + \|\bsh_{\xi'}\|^2_2}
       {\max(\|\bsh_{\xi'}\|^2_2,\,c) \, \|\bsh_{\xi-\xi'}\|^2_2}
  \le
  \max(1/c, \, 4)
  .
\]
Let $c' := \max(1/c,\, 4)$.
We continue from~\eqref{eq:cont} to obtain
\begin{align*}
\max_{\xi'\in\Z_n} \sum_{\substack{\xi=0\\\xi\ne\xi'}}^{n-1} 
     \frac{\|\bsh_\xi\|^2_2}{\max(\|\bsh_{\xi'}\|^2_2,\,c)}
      \frac{\|\bsh_{\xi-\xi'}\|^2_2}{\|\bsh_{\xi-\xi'}\|^2_2} \,|w_{\xi-\xi'}|
    &\le
c' \, 
\max_{\xi'\in\Z_n} \sum_{\substack{\xi=0\\\xi\ne\xi'}}^{n-1}
     \|\bsh_{\xi-\xi'}\|^2_2 \, \Biggl|
     \sum_{\bsh\in A(\bsz,n,\xi-\xi')} \widehat{v}(\bsh) \Biggr|
\\
&= c' \,
\max_{\xi'\in\Z_n}
\sum_{\substack{\xi=0\\\xi\ne\xi'}}^{n-1}
      \|\bsh_{\xi-\xi'}\|^2_2 \, \Biggl|
      \sum_{\bsh\in A(\bsz,n,\xi-\xi')} \widehat{v}(\bsh) \Biggr|
\\
&\le c' \,
\max_{\xi'\in\Z_n}
\sum_{\substack{\xi=0\\\xi\ne\xi'}}^{n-1} 
\sum_{\bsh\in A(\bsz,n,\xi-\xi')} \|\bsh\|^2_2 \, |\widehat{v}(\bsh)|
\\
&\le c' \,
\max_{\xi'\in\Z_n}
\sum_{\bsh\in \Z^d} \|\bsh\|^2_2 \, |\widehat{v}(\bsh)|
  .
\end{align*}
The last inequality follows from~\eqref{eq:disjunct} and is independent from $\xi'$ such that we can drop the maximum.
For the function $v\in E_\alpha(\T^d)$ with $\alpha>5/2$ the following holds by applying the Cauchy--Schwarz inequality and multiplying and dividing by $r_\alpha$, as defined in~\eqref{eq:ralpha},
\begin{align*}
  \sum_{\bsh\in \Z^d} \|\bsh\|^2_2 \, | \widehat{v}(\bsh)|
&\le
  \left( 
    \sum_{\bsh\in \Z^d} r^2_\alpha(\bsh) \, |\widehat{v}(\bsh)|^2
  \right)^{1/2}
  \left(
    \sum_{\bsh\in \Z^d} \frac{\|\bsh\|^4_2}{r^2_\alpha(\bsh)}
  \right)^{1/2}
  \\
&\le
  \|v\|_{E_\alpha(\T^d)}
  \left(
     \sum_{\bsh\in \Z^d} \frac{(\sqrt{d} \, \|\bsh\|_\infty)^4}{r^2_\alpha(\bsh)} 
  \right)^{1/2}
  \\
&\le
  \|v\|_{E_\alpha(\T^d)}
  \left(
    \sum_{\bsh\in \Z^d} \frac{d^2}{r^2_{\alpha-2}(\bsh)}  \right )^{1/2} \\
&\le
 \|v\|_{E_\alpha(\T^d)}
  \left( d^2 \, (1 + 2 \, \zeta(2\alpha-4))^d \right)^{1/2} 
  <
  \infty
  .
\end{align*}
Therefore we have bounded $\|M\|_1$ independent of $n$.
For $\|M\|_\infty$ we can proceed in a similar way to obtain
\begin{align*}
  \|M\|_\infty
  &=
  \max_{\xi\in\Z_n} \sum_{\substack{\xi'=0\\\xi'\ne\xi}}^{n-1}
  \left|
  \frac{\|\bsh_{\xi}\|^2_2 - \|\bsh_{\xi'}\|^2_2}
       {\|\bsh_{\xi'}\|^2_2 + c} \, w_{\xi-\xi'}
  \right|
 \\
  &\le
  \|v\|_{A(\T^d)} + c' \, \|v\|_{E_\alpha(\T^d)} \, \left( d^2 \, (1 + 2 \, \zeta(2\alpha-4))^d \right)^{1/2} 
  .
\end{align*}
Therefore, for any $\bsy \in \R^n$ it holds that
\[
  \|(D\,W-WD) \, \bsy\|_2 
  \le
  c_1 \, \|(D+I) \, \bsy\|_2
,
\]
where $c_1$ is a constant independent of $\bsy$ and $n$.
\\[2mm]
(\textit{ii}) Similar argument holds for second order convergence.
Then
\begin{align*}
 [D,\,[D,W]]\,(D+I)^{-2}
 &=
 \left( 
 \frac{
   (\|\bsh_{\xi}\|^2_2-\|\bsh_{\xi'}\|^2_2)^2
  }{
    2\pi^2\tsc \, (\|\bsh_{\xi'}\|^2_2+c)^2
  }
  \, w_{\xi-\xi'}
 \right)_{\xi,\xi'=0,\ldots,n-1},
\end{align*} 
with the same constant $c=1/(2\pi^2\tsc)$.
For $\xi \ne \xi'$ we can multiply and divide by $\|\bsh_{\xi-\xi'}\|_2^4$ and then
\[
\frac{(\|\bsh_{\xi}\|^2_2-\|\bsh_{\xi'}\|^2_2)^2}{(\|\bsh_{\xi'}\|^2_2+c)^2\|\bsh_{\xi-\xi'}\|_2^4} 
\le
\frac{(\|\bsh_{\xi-\xi'}+\bsh_{\xi'}\|^2_2+\|\bsh_{\xi'}\|^2_2)^2}{(\|\bsh_{\xi'}\|^2_2+c)^2\|\bsh_{\xi-\xi'}\|_2^4} 
\]
has an upper bound of $\max(25,1/c^2)$.
Therefore, the $\ell_1$ and $\ell_\infty$ induced norms of this matrix can be bounded if the potential function $v(\bsx)$ is in Korobov space $E_\alpha(\T^d)$ with the smoothness parameter $\alpha>9/2$:
\begin{align*}
\max_{\xi'\in\Z_n} \sum_{\substack{\xi=0\\\xi\ne\xi'}}^{n-1} \|\bsh_{\xi-\xi'}\|_2^4 \, |w_{\xi-\xi'}| 
&\le
\sum_{\bsh\in\Z^d}\|\bsh\|_2^4 \, |\widehat{v}(\bsh)|
\le
\|v\|_{E_\alpha(\T^d)}
\left(d^4 \, (1+2\zeta(2\alpha-8))^d  \right)^{1/2}
.
\end{align*}
We have hence proved the claims of the lemma.
\end{proof}

In Algorithm~\ref{alg:strang}, our procedure of the time-stepping is shown. Each time-step is done with complexity $\mathcal{O}(n\log n)$.
Matrices $\rme^{-\frac{\imagunit}{2\tsc} V_n  \Deltat}$ and $\rme^{-\frac{\imagunit\tsc}{2} D_n \Deltat}$ are diagonal, hence there is no need to store $n$-by-$n$ matrices.

\begin{algorithm}
\caption{Strang splitting}
\label{alg:strang}
\begin{algorithmic}
\State{Input:}
\State{$\Deltat,m,g$} \Comment{$m\Deltat = T$ is the final time, $g$ is the initial condition}
\State{$\calA(\bsz,n) = \{ \bsh_{0},...,\bsh_{n-1} \} \subset \Z^{d\times n} $} \Comment{Anti-aliasing set of full cardinality}
\State{$\Lambda(\bsz,n) = \{ \bsp_{0},...,\bsp_{n-1} \} \subset \T^{d\times n} $} \Comment{Lattice points}
\State{} 
\State{$V_n = \diag\left( \left( v(\bsp_k) \right)_{k=0,\ldots,n-1} \right)$} \Comment{The potential matrix on lattice points}
\State{$D_n =  \diag\left((4\pi^2\|\bsh_\xi\|_2^2)_{\xi=0,\ldots,n-1}\right)$} \Comment{The Laplacian matrix on the anti-aliasing set}
\State{$\widehat{\bsu}_a^0=\widehat{\bsg}_a=(\widehat{g}_a(\bsh_{0}),\dots,\widehat{g}_a(\bsh_{n-1}))=F_n (g(\bsp_{0}),\dots,g(\bsp_{n-1}))/\sqrt{n} $} 
\Statex{ }
\For{\texttt{$k=1,2,...,m$}} 
        \State $\widehat{\bsu}_a^k=F_n \rme^{-\frac{\imagunit}{2\tsc} V_n  \Deltat}\, F_n^{-1} \rme^{-\frac{\imagunit\tsc}{2} D_n \Deltat}\,F_n \rme^{-\frac{\imagunit}{2\tsc} V_n  \Deltat}\, F_n^{-1} \, \widehat{\bsu}_a^{k-1}$
\EndFor 
\Statex{ }
\State Output: $\widehat{\bsu}_a^m$
\end{algorithmic}
\end{algorithm}

\subsection{Strang splitting and rank-$r$ lattices}
In this section, we generalize the results of the previous section for rank-$r$ lattices. 
Consider a rank-$r$ lattice $\Lambda(\bsZ,\bsn)$ in canonical form, and the corresponding anti-aliasing set $\calA(\bsZ,\bsn)$ with full cardinality $n=\prod_{i=1}^r n_i$.
We enumerate the anti-aliasing set in ``lexicographical ordering'' by identifying $\bsh^{(\chi)} = \bsh_{\bsxi}$ for $\chi=0,\ldots,n-1$ and $\bsxi = \bsh_{\bsxi} \cdot \bsz \bmod{\bsn}$ for all $\bsh_{\bsxi} \in \calA(\bsZ,\bsn)$ such that 
\begin{align*}
\chi
=
\xi_1 \, n_2 \cdots n_r + \xi_2 \, n_3 \cdots n_r + \cdots + \xi_r 
=
\sum_{i=1}^{r} \left( \xi_i \prod_{j=i+1}^{r} n_j \right)
,
\end{align*}
for all $\bsxi \in \Z_{n_1} \oplus \cdots \oplus \Z_{n_r}$.
Likewise, we enumerate the lattice points by identifying $\bsp^{(\kappa)} = \bsp_{\bsk}$ for $\kappa = 0,\ldots,n-1$ such that
\[
\kappa
=
\kappa_1 \, n_2 \cdots n_r + \kappa_2 \, n_3 \cdots n_r + \cdots + \kappa_r 
=
\sum_{i=1}^{r} \left( \kappa_i \prod_{j=i+1}^{r} n_j \right)
,
\]
for all $\bsk \in \Z_{n_1} \oplus \cdots \oplus \Z_{n_r}$.
Then the ordinary differential equation \eqref{ode} holds with,
$\widehat{\bsu}_{0} =\widehat{\bsg}_a:=(\widehat{g}_a(\bsh^{(0)}),\dots,\widehat{g}_a(\bsh^{(n-1)}))$,
\begin{align}\label{eq:Dn_r}
  D_\bsn
  :=
  \diag\left((4\pi^2\|\bsh^{(\chi)}\|_2^2)_{\chi=0,\ldots,n-1}\right)
  ,
\end{align}
and 
\begin{align}\label{eq:Wn_r}
 W_\bsn := F_\bsn V_\bsn F^{-1}_\bsn,
 \end{align}
with 
\begin{align}\label{eq:Vn_r}
  V_\bsn
  :=
  \diag\left( \left( v(\bsp^{(\kappa)}) \right)_{\kappa=0,\ldots,n-1} \right)
  ,
\end{align}
where $F_\bsn$ is the $r$-dimensional discrete Fourier transform. With these notations we have the following generalization of Lemma~\ref{lem:rank-1-commutator-bounds}.

\begin{lemma}[Rank-$r$ commutator bounds]\label{lem:rank-r-commutator-bounds}
Given a rank-$r$ lattice $\Lambda(\bsZ,\bsn)$ in canonical form with the number of points $n=\prod_{i=1}^r n_i$, and a TDSE with a potential function $v \in E_\alpha(\T^d)$ with $\alpha \ge 2$ and an initial condition $g \in E_\beta(\T^d)$ with $\beta \ge 2$.
Let $D = \tfrac{\tsc}2 D_\bsn$ and $W = \frac1\tsc W_\bsn$ with $D_\bsn$ and $W_\bsn = F_\bsn V_\bsn F^{-1}_\bsn$ as defined in~\eqref{eq:Dn_r} and~\eqref{eq:Wn_r}, and with $V_\bsn$ as defined in~\eqref{eq:Vn_r} using the potential function $v$.

If the anti-aliasing set 
$\calA(\bsZ,\bsn) = \{ \bsh_{\bsxi} \in \Z^d : \bsZ^{\top}\bsh_{\bsxi} \equiv \bsxi \pmod{\bsn} \text{ for } \bsxi \in \Z_{n_1} \oplus \cdots \oplus \Z_{n_r} \}$,
with full cardinality, is chosen such that each $\bsh_{\bsxi}$ with $\bsxi \in \Z_{n_1} \oplus \cdots \oplus \Z_{n_r}$ has minimal $\ell_2$ norm, i.e.,
\begin{align*}
  \| \bsh_{\bsxi} \|_2
  =
  \min_{\bsh' \in A(\bsZ,\bsn,\bsxi)} \|\bsh'\|_2
  ,
\end{align*}
with
\begin{align*}
  A(\bsZ,\bsn,\bsxi)
  :=
  \bigl\{ \bsh \in \Z^d : \bsZ^{\top}\bsh  \equiv \bsxi \pmod{\bsn} \bigr\}
  ,
\end{align*}
then the following hold.
\\[2mm]
(\textit{i}) If $v \in E_\alpha(\T^d)$ with parameter $\alpha > 5/2$ then, for all $\bsy \in \R^n$ we have
\begin{align*}
 &\|[D,W] \, \bsy\|_2\le c_1 \|(D+I) \, \bsy\|_2,
 \end{align*}
 where $c_1$ is a constant independent of $\bsn$ and $\bsy$.
\\[2mm]
(\textit{ii}) If $v \in E_\alpha(\T^d)$ with parameter $\alpha > 9/2$ then, for all $\bsy \in \R^n$ we have
 \begin{align*}
 &\|[D,[D,W]] \, \bsy\|_2\le c_2 \|(D+I)^2 \, \bsy\|_2,
\end{align*}
where $c_2$ is a constant independent of $\bsn$ and $\bsy$. 
\end{lemma}
\begin{proof}
Due to the lexicographical ordering on matrices $D_\bsn, W_\bsn$ and $V_\bsn$, we operate in the same way as in Lemma~\ref{lem:rank-1-commutator-bounds}.
\end{proof}

We note that Algorithm~\ref{alg:strang} works in the same manner by replacing the inputs to the rank-$r$ setting and using $r$-dimensional FFTs.

\subsection{Total time discretization error bound}

Combining Theorem~\ref{thm:strang-local-error} with Lemmas~\ref{lem:rank-1-commutator-bounds} and~\ref{lem:rank-r-commutator-bounds} we obtain the following global error bound.

\begin{theorem}[Total error bounds]\label{thm:error-bound}
Given a rank-$r$ lattice $\Lambda(\bsZ,\bsn)$ in canonical form with number of points $n=\prod_{i=1}^r n_i$, and a TDSE with a potential function $v \in E_\alpha(\T^d)$ with $\alpha \ge 2$ and an initial condition $g \in E_\beta(\T^d)$ with $\beta \ge 2$.
Let $D = \tfrac{\tsc}2 D_\bsn$ and $W = \frac1\tsc W_\bsn$ with $D_\bsn$ and $W_\bsn = F_\bsn V_\bsn F^{-1}_\bsn$ as defined in~\eqref{eq:Dn_r} and~\eqref{eq:Wn_r}, and with $V_\bsn$ as defined in~\eqref{eq:Vn_r} using the potential function $v$.

If the anti-aliasing set 
$\calA(\bsZ,\bsn) = \{ \bsh_{\bsxi} \in \Z^d : \bsZ^{\top}\bsh_{\bsxi} \equiv \bsxi \pmod{\bsn} \text{ for } \bsxi \in \Z_{n_1} \oplus \cdots \oplus \Z_{n_r} \}$,
with full cardinality, is chosen such that each $\bsh_{\bsxi}$ with $\bsxi \in \Z_{n_1} \oplus \cdots \oplus \Z_{n_r}$ has minimal $\ell_2$ norm, i.e.,
\begin{align*}
  \| \bsh_{\bsxi} \|_2
  =
  \min_{\bsh' \in A(\bsZ,\bsn,\bsxi)} \|\bsh'\|_2
  ,
\end{align*}
with
\begin{align*}
  A(\bsZ,\bsn,\bsxi)
  :=
  \bigl\{ \bsh \in \Z^d : \bsZ^{\top}\bsh  \equiv \bsxi \pmod{\bsn} \bigr\}
  ,
\end{align*}
then, by applying the Strang Splitting
\begin{equation*}
  \widehat{\bsu}_a^{k+1}
  =
  \rme^{-\frac{\imagunit}{2\tsc} W_\bsn \Deltat}\,\rme^{-\frac{\imagunit\tsc}{2} D_\bsn \Deltat}\,\rme^{-\frac{\imagunit}{2\tsc} W_\bsn\Deltat} \, \widehat{\bsu}_a^k
  \qquad \text{for } k=0,1,\dots,m-1,
\end{equation*}
the following hold:
\\[2mm]
(\textit{i})  If $v \in E_\alpha(\T^d)$ with parameter $\alpha > 5/2$, then the error is bounded for $t=k \Deltat$ by
\begin{align*}
 \| u_a^k(\cdot) - u_a(\cdot,t) \|_{L_2}
 &\le
 \Deltat \; C_1 t \max_{0\le t'\le t} \|(D+I) \, \widehat{\bsu}_{t'}\|_2,
 \end{align*}
 where $C_1$ is a constant independent of $n$, $k$, and $\Deltat$.
\\[2mm]
 (\textit{ii})  If $v \in E_\alpha(\T^d)$ with parameter $\alpha>9/2$, then the error is bounded for $t=k\Deltat$ by
 \begin{align*}
  &\| u_a^k(\cdot)-u_a(\cdot,t) \|_{L_2}
  \le
  (\Deltat)^2 \; C_2 t  \max_{0\le t'\le t} \|(D+I)^2 \, \widehat{\bsu}_{t'}\|_2,
\end{align*}
where $C_2$ is a constant independent of $n$, $k$ and $\Deltat$. 
\label{thm:grobalerror}
\end{theorem}

\begin{proof}
\mbox{}
\\[2mm]
(\textit{i}) Let us denote the Strang splitting operator by $S=\rme^{-\frac{\imagunit}{2\tsc} W_\bsn \Deltat}\rme^{-\frac{\imagunit\tsc}{2} D_\bsn \Deltat }\rme^{-\frac{\imagunit}{2\tsc} W_\bsn \Deltat }$ and the true solution operator by $T=\rme^{-(\frac{\imagunit}{\tsc} W_\bsn -\frac{\imagunit\tsc}{2} D_\bsn )\Deltat}$.
We have the following for first order convergence:
\begin{align*}
  \| u_a^k(\cdot)-u_a(\cdot,k\Deltat) \|_{L_2}
  &=
  \| S^k \, \widehat{\bsg} - T^k \, \widehat{\bsg} \|_2 \\
  &=
  \left\|\sum_{j=0}^{k-1}S^{k-j-1}(S-T)T^{j} \, \widehat{\bsg} \right\|_2 \\
  &\le
  \biggl(\max_{0\le t'\le t}\left \|(S-T) \, \widehat{\bsu}_{t'} \right\|_2 \biggr) \sum_{j=0}^{k-1} \left\| S^{k-j-1} \right\|_2 ,
\end{align*}
by using a telescoping sum, for the maximum argument $t'$ there exists $\ell\in\{0,1,...,k\}$ such that $t'=\ell\Deltat$. 
Applying Theorem~\ref{thm:strang-local-error} and Lemmas~\ref{lem:rank-1-commutator-bounds} and~\ref{lem:rank-r-commutator-bounds}, we have $\|S\bsy-T\bsy\|\le C_1 (\Deltat)^2 \|(D+I)\bsy\|$ for the first order convergence condition and, 
 $\|S\bsy-T\bsy\|\le C_2 (\Deltat)^3 \|(D+I)^2\bsy\|$ for the second order convergence condition, for all $\bsy \in \R^d$. 
Note that $\|S\|_2\le \|\rme^{-\frac{\imagunit}{2\tsc} W_\bsn \Deltat}\|_2\|\rme^{-\frac{\imagunit\tsc}{2} D_\bsn \Deltat}\|_2\|\rme^{-\frac{\imagunit}{2\tsc} W_\bsn \Deltat}\|_2=1$, because the $\ell_2$ norm of a matrix is the largest singular value of the matrix, e.g.,
\begin{align*}
  \left\|\rme^{-\frac{\imagunit}{2\tsc} W_\bsn \Deltat}\right\|_2
  &=
  \sqrt{\lambda_{\max} \left(\rme^{-\frac{\imagunit}{2\tsc} W_\bsn \Deltat} \left(\rme^{-\frac{\imagunit}{2\tsc} W_\bsn \Deltat}\right)^* \right)}\\
  &=
  \sqrt{\lambda_{\max} \left(F_\bsn J_{v,\bsn} F^{-1}_\bsn \left( F_\bsn J_{v,\bsn} F^{-1}_\bsn \right)^* \right)}\\
  &=
  \sqrt{\lambda_{\max} (I)}=1,
\end{align*}
where $\lambda_{\max}$ denotes the largest eigenvalue, $J_{v,\bsn} = \diag[(\rme^{-\frac{\imagunit}{2} v(\bsx) \Deltat})_{\bsx \in \Lambda(\bsZ, \bsn)}]$ and $A^*$ is the Hermitian conjugate of $A$.
Hence we obtain the following for the first case:
\begin{align*}
  \Biggl( \sum_{j=0}^{k-1} \| S^{k-j-1} \|_2 \Biggr) \max_{0\le t'\le t}\|(S-T)\, \widehat{\bsu}_{t'}\|_2&
  \le
  k C_1 (\Deltat)^2 \max_{0\le t'\le t}\|(D+I) \, \widehat{\bsu}_{t'}\|_2 \\
  &=
  C_1 t \Deltat  \max_{0\le {t'}\le t} \|(D+I) \, \widehat{\bsu}_{t'}\|_2
  .
\end{align*}
\\[2mm]
(\textit{ii}) For second order convergence a similar argument holds and we obtain
\begin{align*}
  \left\| u_a^k(\cdot)-u_a(\cdot,k\Deltat) \right\|_{L_2}
  &\le
  k C_2 (\Deltat)^3 \max_{0\le t'\le t}\left\|(D+I)^2 \, \widehat{\bsu}_{t'}\right\|_2 \\
  &=
  C_2 (\Deltat)^2 t \max_{0\le t'\le t} \left\|(D+I)^2 \, \widehat{\bsu}_{t'}\right\|_2
  .
\end{align*}
This concludes the proof.
\end{proof}

Note that this shows that the smoothness for the potential $v$ required for second order convergence is independent of the number of dimensions.
This is a big improvement compared to the results shown in \cite{G07} with respect to sparse grids, where the smoothness $\alpha$ needs to increase for increasing dimension to obtain second order convergence.

\section{Numerical results}\label{sec:3}

In this section, we demonstrate the method with numerical results. We particularly consider three quantities of interest: approximation error against the time step;
evolution of the norm and the energy of the wave function over the time period; and the error which is caused by the physical discretization.
To compare with the results from \cite{G07} using sparse grids, we choose the same experiments, but since our method allows the results to also be calculated for higher $d$ than in \cite{G07} we extended the experiments.

\subsection{Component-by-component construction}

For constructing the rank-$1$ lattice and the anti-aliasing set, we employ the fast component-by-component construction for lattice sequences, see, e.g., \cite{MR2272256}.
We use the script \texttt{fastrank1expt.m}, available online \cite{MR2198499} for fast component-by-component construction of a rank-$1$ lattice sequence with a prime power of points. We use powers of~$2$.
The lattice point set is optimized for integration in the (unweighted) Korobov space with smoothness $\alpha=1$ (in a common alternative notation this is $\alpha=2$, as is the case for the construction script).
After having obtained the generating vectors we construct the corresponding anti-aliasing sets in accordance with Lemma~\ref{lem:rank-1-commutator-bounds} in the following manner:
\begin{enumerate}
\itemsep0em
\item Generate all $\bsh \in \Z^d$ for which $\|\bsh \|_2 \le r$ for some well chosen $r$.
\item Sort the points according to the $\ell_2$-norm in ascending order.
\item Calculate $m_{\bsh} \equiv \bsh \cdot \bsz  \pmod{n}$ in sorted order and add $\bsh$ to $\calA(\bsz,n)$ if the value $m_{\bsh}$ has not been seen before. Repeat this step until the set has the cardinality $n$.
\end{enumerate}
We refer to \cite[Section~2.6]{cools2010constructing} for iteratively constructing $\bsh$ in a bounded region.

To compare our results with the results in \cite{G07}, we regenerated the data from that paper as accurately as possible from the graphs therein.
In Figures \ref{timestep1} and~\ref{timestepHigh}, we denote with SG the results from \cite{G07} using sparse grids, and by LR our method using lattice rules.
To make a fair comparison, we choose as close as possible the same number of basis functions $n$ as in \cite{G07} whenever this is known.
We calculate the number of basis functions $n_{GS}$ for the $d$-dimensional sparse grid with level $\ell$ by
\[
n_{\textrm{SG}}
=
\sum_{i=0}^{\ell-1} 2^i \binom{d-1+i}{i}. 
\]  
The corresponding numbers of basis functions for both methods and the generating vectors for the rank-$1$ lattice used in the experiments are exhibited in Table~\ref{tb:parameter}.

\begin{table}[]
\centering
\label{tb:parameter}
\begin{tabular}{c|c|c|c}\hline
 $d$ & $n$  & $\bsz^{\top}$& $n_{\textrm{SG}}$ from \cite{G07} \\\hline
\vphantom{$A^{A^A}$} $2$  & $2^{18} $& $(1,100135)$ &  $2^{17.7}$ or $2^{19.9}$ \;*   \\[2mm]
 & $2^{20}$ & $ (1, 443165)$& $2^{19.9}$ \\[2mm]
 $3$  &$2^{22}$ & $(1,1737355,261247)$ & $2^{22.9}$   \\[2mm]
  & $2^{25}$ & $ (1,12386359,15699201)$& $2^{25.4}$ \\[2mm] 
  $4$ to $12$  & $2^{25}$ & \begin{tabular}{@{}c@{}}$(1, 12386359,15699201,6807287,$\\	$13966305, 6107923, 4432603, 2304135$\\	$7323801, 5705679, 5643703, 3867405)$\end{tabular} & Not available\\[2mm]\hline
\end{tabular}
\caption{Parameters of the numerical results. For (*) the level of the sparse grid is not specified as one number in \cite{G07}. For $d\ge 4$, we always choose $n=2^{25}$, and $\bsz$ is chosen to be the first $d$ components, e.g., for $d=4$, $\bsz^{\top}=(1, 12386359,15699201,6807287)$.}
\end{table}

\subsection{Convergence with respect to time step size}

As is in \cite{G07_1,G07,jahnke2000error} we consider the error of the calculated solution in terms of decreasing time steps against a reference solution.
We choose two types of the initial condition $g$ from \cite{G07}, the ``Gaussian'' initial condition given by:
\[
  g_1(\bsx)
  :=
  \left(\frac{2}{\pi \tsc}\right)^{d/4}\exp\left(-\frac{\left(2\pi x_1 - \frac{3\pi}{2}\right)^2 + \sum_{j=2}^{d}\left(2 \pi x_j-\pi\right)^2}{\tsc}\right)\frac{1}{c_{1}},
\]
and the ``Hat'' initial condition given by:
\[
  g_2(\bsx)
  :=
  \left(\frac{3}{\pi \sqrt{\tsc}}\right)^{d/2}   \left(1-\frac{2}{\pi \sqrt{\tsc}} \left|2 \pi x_1 - \frac{3\pi}{2}\right|\right) \; \prod_{j=2}^{d}\left(1-\frac{2}{\pi \sqrt{\tsc}}\left|2 \pi x_j - {\pi}\right|\right)\frac{1}{c_{2}},
\]
for $\bsx\in [0,1)^d$ where $c_{1}$ and $c_{2}$ are normalizing constants to make the $L_2$ norms of both functions equal to 1. We remark that in \cite{G07}, the domain was erroneously stated as $[-\pi,\pi]^d$ which would be equivalent to $[-1/2,1/2)^d$ in our setting. However, we conclude that the actual calculation was done in $[0,2\pi]^d$, as can be confirmed by the fact that the calculated norm of the Gaussian function was 1 in \cite[Figure~6.8]{G07} therein,
and the fact that the same author has exactly the same result in another paper \cite{G07_1} where the domain is stated as $[0,2\pi]^d$ with the same Gaussian initial condition, which corresponds to $[0,1)^d$ in our case.
Therefore we conclude that our experiment is the same experiment as in \cite{G07}.
For the potential function $v$, we consider a ``smooth'' potential function
\[
  v_1(\bsx)
  =
  \prod_{i=1}^d (1-\cos(2\pi x_j))
  ,
\]
and a ``harmonic'' potential function
\[
  v_2
  =
  \frac{1}{2}\sum^d_{j=1}(2\pi x_j-\pi)^2
  .
\]
 
To show the time discretization error $\|u_a(\bsx,t)-u_a^m (\bsx)\|_{L_2}$ at time $t=m\,\Deltat=1$ being fixed, we calculate a reference solution $u^{M}_a(\bsx)$ with the finest time step size $\Deltat=1/M=1/10000$, as an approximation of $u_a(\bsx,t)$. Then we calculate $u^{m}_a(\bsx)$ with various time step sizes $\Deltat=1/m=1/5,...,1/1000$ to be able to plot the convergence rate of $\|u^M_a(\bsx)-u_a^{m}(\bsx)\|_{L_2}$.

The result is exhibited in Figures \ref{timestep1} and~\ref{timestepHigh}.
We observe that the convergence rate for our new method consistently shows second order convergence $\mathcal{O}((\Deltat)^{2})$.
On the other hand the sparse grid results from \cite{G07} do not; for instance, see the case $d=3$ with $\tsc=0.01$.
We remark that the initial condition $g_1$ combined with the potentials $v_1$ and $v_2$ satisfy the conditions of Lemma~\ref{lem:u-series} and Theorem~\ref{thm:grobalerror}.
Therefore we expect to see second order convergence in those cases. However, the hat initial condition $g_2$ does not satisfy the required regularity, nevertheless we have second order convergence in all cases.
Moreover, our method achieves the second order convergence consistently even for high-dimensional cases, going from $d=4$ in Figure~\ref{timestep1} up to $d=12$ in Figure~\ref{timestepHigh}.
We note that for $d=10$ and $d=12$ the convergence graph for the potential $v_1$ does show some irregular behaviour. This comes from the numerical exuberance of the function $v_1$ itself when the dimension is high;
the function rapidly increases to $2^d$ when the position $\bsx$ is close to $(1/2,\ldots,1/2)$. This phenomenon does not happen with the harmonic potential $v_2$, which is more relevant for physics applications.

\begin{figure}
\centering

  \includegraphics{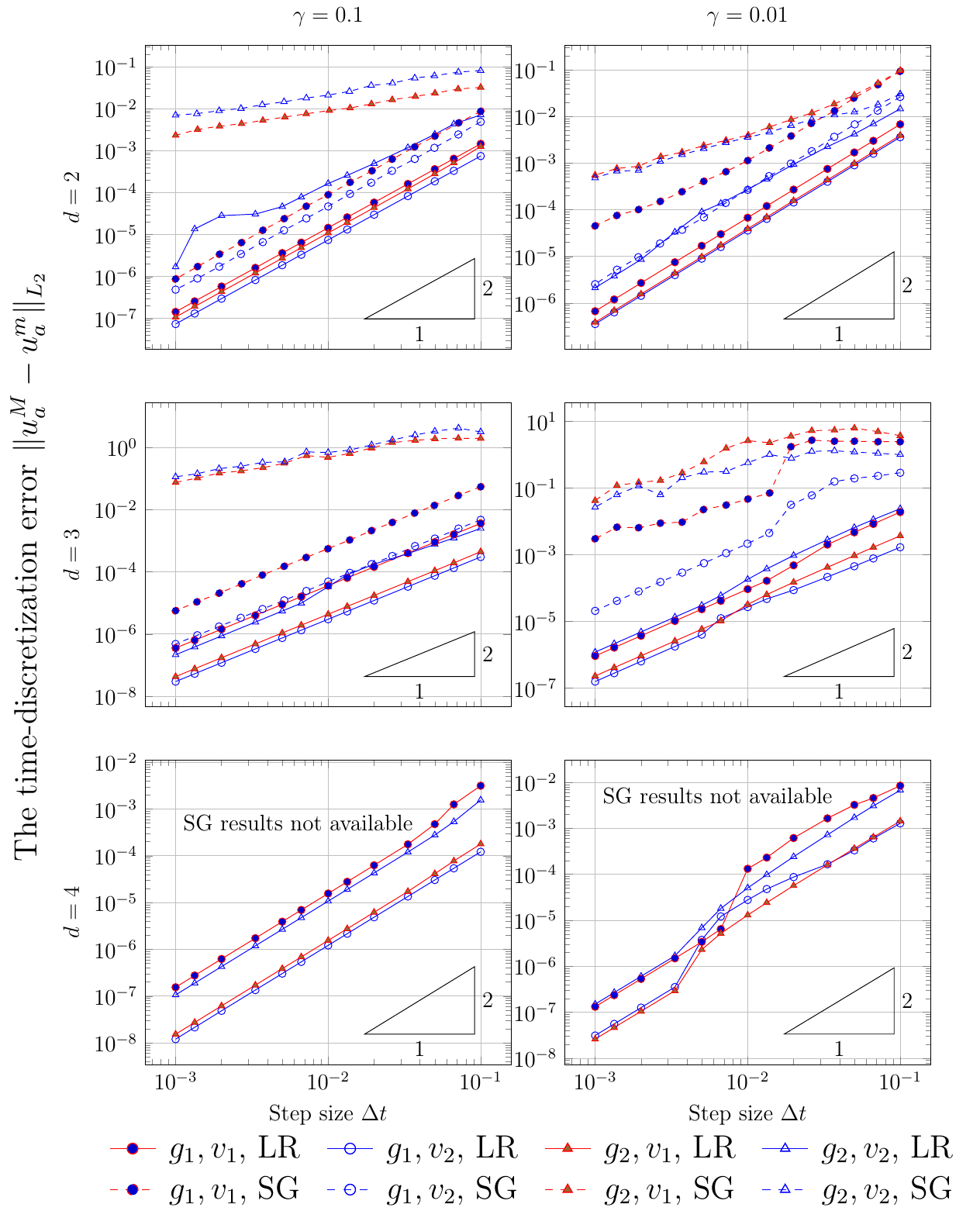}

 \caption{The time-discretization error. Our method (LR) is presented by the solid line, and the results by sparse grid (SG) from \cite{G07} by the dotted line. Note that the initial condition $g_2$ does not satisfy the regularity condition.}
 \label{timestep1}
\end{figure}

\begin{figure}
\centering

  \includegraphics{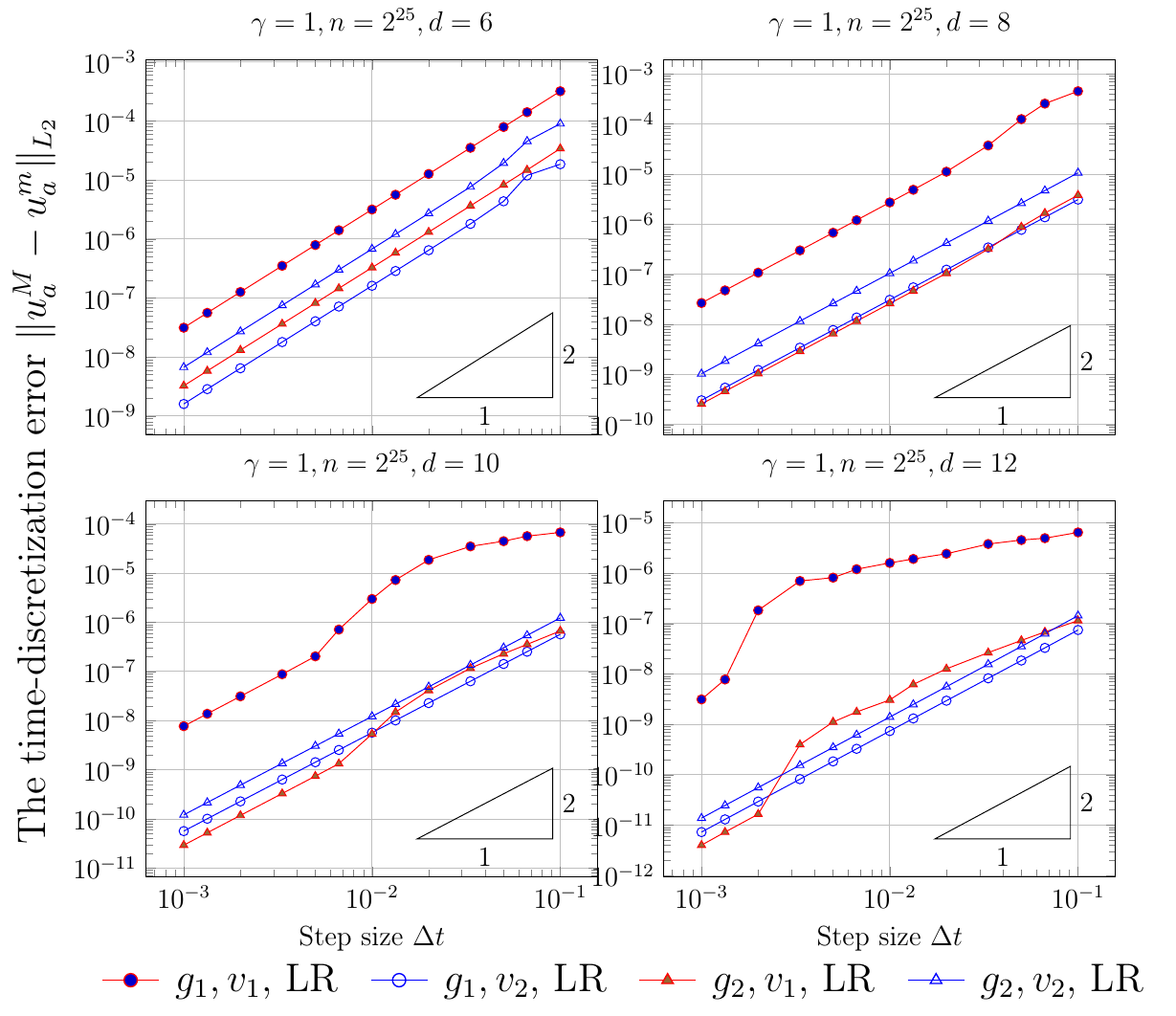}

\caption{The time-discretization error in high-dimensional cases. Results by sparse grid \cite{G07} is not available for these higher-dimensional cases. }
 \label{timestepHigh}
 \end{figure}

\subsection{Norm and energy conservation}\label{subsec:conservation}

The TDSE, as a physical system, needs to conserve the norm and energy of the system.
To test our algorithm we look at how well these quantities are preserved numerically.
Denote the Hamiltonian by $H := -\frac{1}{2} \tsc\, \nabla^2 +\frac{1}{\tsc}v$, then $\frac{\partial u}{\partial t}=-\imagunit H u.$
We study the time evolution of the $L_2$ norm of the wave function $\|u_a\|_{L_2}$ and the energy
 $\langle{Hu_a},{u_a}\rangle_{L_2}$, where $\langle\,,\,\rangle_{L_2}$ denotes the Hermitian inner product in the $L_2$ space. 
These two quantities are supposed to be conserved over the time period since
\[ 
  \frac{\partial}{\partial t} \langle{u},{u}\rangle_{L_2}
  =
  \langle -\imagunit Hu, u\rangle_{L_2} + \langle u, -\imagunit Hu \rangle_{L_2}
  =
  0
  ,
\]
and 
\[
  \frac{\partial}{\partial t}  \langle Hu, u \rangle_{L_2} 
  =
  \langle -\imagunit Hu, Hu \rangle_{L_2} + \langle Hu, -\imagunit Hu\rangle_{L_2}
  =
  0
  ,
\]
for the self-adjoint Hamiltonian $H$. For the self-adjointness of the Hamiltonian, we refer to \cite{MR0493420}.
Our numerical results are presented in Figure~\ref{fig:conservation}.
To compare with the result from \cite{G07}, we traced the graph therein, but we also need to remark that the absolute value in there was not accurate; the axis of the graph in \cite{G07} is not informative enough for this purpose.
However, since the value of $(\mathrm{max}-\mathrm{min})/\mathrm{mean}$ was exhibited in the article, we can compare the variation. Therefore, we plot the time-evolution of the norm and the energy where the initial values are adjusted to zero.

\begin{figure}
\centering

  \includegraphics{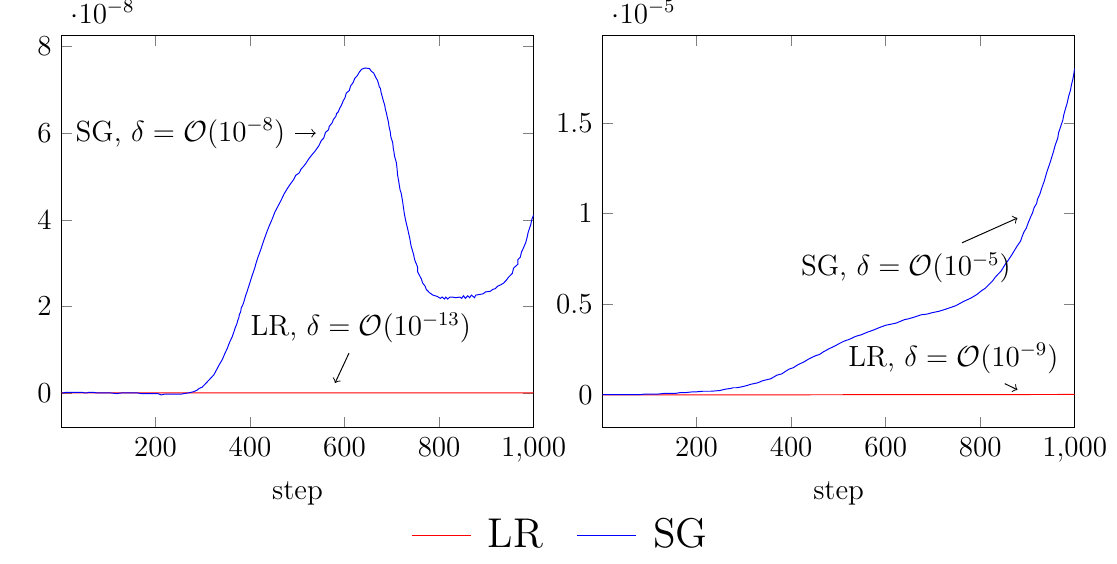}

    \caption{Variation of the norm (left) and the energy (right) for $\tsc=0.5,d=5$ with $g_1$ and $v_2$.
    }
\label{fig:conservation}
\end{figure}

In Figure~\ref{fig:conservation} we see the two quantities are conserved much more accurately using our algorithm than when using the sparse grid approach in \cite{G07}.
We calculate the quantity $\delta := (\mathrm{max}-\mathrm{min})/\mathrm{mean}$ to give an indication of the variation.
Our method conserves more accurately than the sparse grid approach, for the norm conservation we have a factor of $10^{-5}$ smaller variation and for the energy conservation we have a factor of $10^{-4}$.
The reason of the stability of our method is coming from the unitarity of the Fourier transform on our lattice points.
Due to unitarity, the potential operator in the frequency domain, $F_\bsn V_\bsn F^{-1}_\bsn$, becomes Hermitian. Therefore the operator matrix $F_\bsn V_\bsn F^{-1}_\bsn+D_\bsn$ is also Hermitian and hence the spectral theorem tells us that the eigenvalues of the operator matrix are all real.
Finally, the time evolution operator is norm and energy conserving, i.e., $\| \rme^{ -\frac{\imagunit}{\tsc} W_\bsn \, t   - \frac{\imagunit\tsc}{2} D_\bsn \, t } \|_2 = 1$.
In contrast, the Fourier transform on the sparse grid in \cite{G07} is not unitary.
The lack of unitarity can lead to numerical issues and can even lead to have the exponential error growth, instead of linear, in time \cite[Section~III.1.4]{MR2474331}. 

\subsection{Discussion on the the initial discretization}

Here we study the initial error which is caused by the initial discretization in space. 
The total mean square error of the initial (spatial) discretization is given by
\begin{align*}
 e^2_{\text{total}} 
 &=
 \|g-g_{a}\|_{L_2}^2 \\
  &=
  {\displaystyle\int_{[0,1]^s}\bigg|\sum_{\bsh \in \Z^d}\widehat{g}(\bsh)\,\exp(\twopii\bsh\cdot\bsx)-\sum_{\bsh \in \calA(\bsz,n)}\widehat{g}_a(\bsh)\,\exp(\twopii\bsh\cdot\bsx)\bigg |^2 \rd{\bsx}} \\
  &=
  {\displaystyle\sum_{\bsh \in \Z^d\setminus \calA(\bsz,n)}|\widehat{g}(\bsh)|^2} + {\displaystyle\sum_{\bsh \in \calA(\bsz,n)}|\widehat{g}(\bsh)-\widehat{g}_a(\bsh)|^2}.
\end{align*}

\begin{figure}
\centering

  \includegraphics{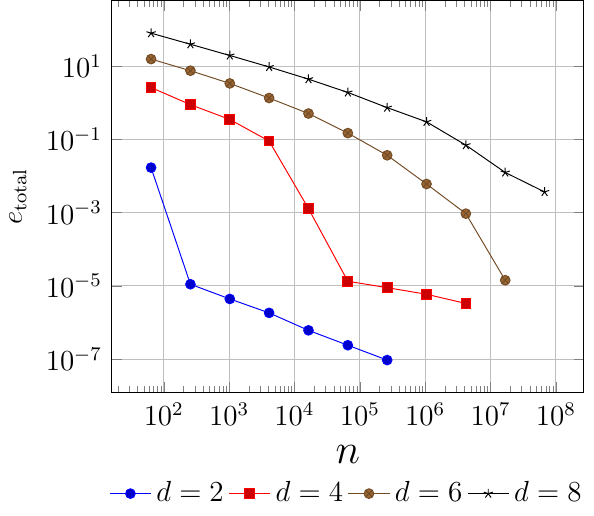}

\caption{The initial discretization error $e_{\text{total}}$ for $\tsc=1$ with Gaussian initial condition $g_1$. }
\label{initial}
\end{figure}
We plot the error $e_{\text{total}}$ in Figure~\ref{initial} with different dimensionality for the Gaussian initial condition. Approximating functions still requires many basis functions when the dimension becomes higher. However, intuitively we might argue that our way of choosing the basis functions according to the $\ell_2$ distance works well particularly for the Gaussian initial condition since the magnitude of the Fourier coefficients of a Gaussian is also a Gaussian (i.e., only depends on the $\ell_2$ norm of the frequency, and decays exponentially fast).

\section{The total error of full discretization}\label{sec:4}
The total error of the method is coming from the discretization both in space and time.
Here we recall our notation for approximating the solution:
\begin{enumerate}
 \item $u(\bsx,t)$ is the true solution of \eqref{TDSE};
 \item $u_a (\bsx,t)$ is the spatially discretized solution including the dynamics as \eqref{ode};
 \item $u_a^k(\bsx)$ is the fully discretized solution with Strang splitting \eqref{StrangTDSE}.
\end{enumerate}
First we denote by $\mathcal{I}_{\bsn}$ the interpolation operator on the lattice points, for a function $f$,
\[
 \mathcal{I}_\bsn (f) (\bsx,t) 
 :=
 \sum_{\bsh\in \calA({\bsZ,\bsn})} \widehat{f}_a(\bsh,t) \, \exp(\twopii\bsh\cdot\bsx)
\]
where,
\[
  \widehat{f}_a(\bsh,t)
  :=
  {\frac1n}
  \sum_{\bsp\in \Lambda({\bsZ,\bsn})}
    f(\bsp,t) \, \exp(-\twopii\bsh\cdot \bsp).        
\]
By using the interpolation operator, we can bound
\[
\|u(\cdot,t)-u_a (\cdot,t)\|_{L_2} 
\le
\|u(\cdot,t)-\mathcal{I}_{\bsn} (u) (\cdot, t)\|_{L_2} + \|\mathcal{I}_{\bsn} (u) (\cdot, t)-u_a (\cdot,t)\|_{L_2}. 
\]
The error $\|u_a (\cdot,t)-u_a^k(\cdot)\|_{L_2}$ is already bounded by Theorem~\ref{thm:grobalerror}. 
Using the triangle inequality we can then bound the total error.
\begin{theorem}[Total error]\label{thm:total}
Given a rank-$r$ lattice $\Lambda(\bsZ,\bsn)$ in canonical form with the number of points $n=\prod_{i=1}^r n_i$, and a TDSE with a potential function $v \in E_\alpha(\T^d)$ with $\alpha \ge 9/2$ and an initial condition $g \in E_\beta(\T^d)$ with $\beta \ge 2$.
Let $D = \tfrac{\tsc}2 D_\bsn$ and $W = \frac1\tsc W_\bsn$ with $D_\bsn$ and $W_\bsn = F_\bsn V_\bsn F^{-1}_\bsn$ as defined in~\eqref{eq:Dn_r} and~\eqref{eq:Wn_r}, and with $V_\bsn$ as defined in~\eqref{eq:Vn_r} using the potential function $v$.

If the anti-aliasing set 
$\calA(\bsZ,\bsn) = \{ \bsh_{\bsxi} \in \Z^d : \bsZ^{\top}\bsh_{\bsxi} \equiv \bsxi \pmod{\bsn} \text{ for } \bsxi \in \Z_{n_1} \oplus \cdots \oplus \Z_{n_r} \}$,
with full cardinality, is chosen such that each $\bsh_{\bsxi}$ with $\bsxi \in \Z_{n_1} \oplus \cdots \oplus \Z_{n_r}$ has minimal $\ell_2$ norm, i.e.,
\begin{align}\label{eq:minimal_r}
  \| \bsh_{\bsxi} \|_2
  =
  \min_{\bsh' \in A(\bsZ,\bsn,\bsxi)} \|\bsh'\|_2
  ,
\end{align}
with
\begin{align*}
  A(\bsZ,\bsn,\bsxi)
  :=
  \bigl\{ \bsh \in \Z^d : \bsZ^{\top}\bsh  \equiv \bsxi \pmod{\bsn} \bigr\}
  ,
\end{align*}
then the following bound holds:
\begin{align*}
&\|u(\cdot,t)-u_a ^k (\cdot)\|_{L_2} \\
&\le
\|u(\cdot,t)-\mathcal{I}_{\bsn} (u) (\cdot, t)\|_{L_2} + \|\mathcal{I}_{\bsn} (u) (\cdot, t)-u_a (\cdot,t)\|_{L_2}
+ \|u_a (\cdot,t)-u_a ^k (\cdot)\|_{L_2}  \\
& \le
 2 \sum_{\bsh \in \Z^d\setminus \calA(\bsZ,\bsn)}| \widehat{u}(\bsh,t) | + \frac{t\tsc}{2}\max_{0 \le t'\le t}\sum_{\bsh \in \Z^d\setminus \calA(\bsZ,\bsn)} \|\bsh\|_2^2 \; | \widehat{u}(\bsh,t') |\\
&\kern4em+ (\Deltat)^2 \; C_2 t  \max_{0\le t'\le t} \|(D+I)^2 \, \widehat{\bsu}_{t'}\|_2,
\end{align*}
where $C_2$ is a constant independent of $n$, $k$ and $\Deltat$.
\end{theorem}
\begin{proof}
 To show the error, we follow a similar way of the proof for \cite[Theorem 1.8]{MR2474331} where the one-dimensional pseudo-spectral Fourier method for the TDSE is analyzed.
 Applying the interpolation operator to \eqref{TDSE} on both sides, we have
\begin{align}\label{eq:tdse_intp}
      \frac{\partial \mathcal{I}_{\bsn} (u) (\bsx, t)}{\partial t}
    &=
    \frac{ \tsc \imagunit }{2} \, \mathcal{I}_{\bsn} (\nabla^2 u) (\bsx, t) - \frac{\imagunit}{\tsc} \mathcal{I}_{\bsn} ( v \, u )(\bsx,t) \nonumber \\
    &= \frac{ \tsc \imagunit}{2} \, (\nabla^2  \mathcal{I}_{\bsn} (u) ) (\bsx, t)  - \frac{\imagunit}{\tsc} \mathcal{I}_{\bsn} ( v \,  (\mathcal{I}_{\bsn}(u) ) ) (\bsx,t) + \delta_{\bsn}(\bsx, t),
\end{align}
where $\delta_{\bsn} (\bsx, t)=\frac{ \tsc \imagunit }{2} \, \mathcal{I}_{\bsn} (\nabla^2 u(\bsx, t)) - \frac{ \tsc \imagunit}{2} \, (\nabla^2  \mathcal{I}_{\bsn} (u) (\bsx, t))$ is called the \emph{defect} which can be seen as a commutator of the interpolation operator and the Laplacian applied to the solution,
and we used $\mathcal{I}_{\bsn} ( v \, u )(\bsx,t) = \mathcal{I}_{\bsn} ( v \,  (\mathcal{I}_{\bsn}( u ) ) ) (\bsx,t)$.
At the same time, we can express the dynamics of $\widehat{u}_a(\bsh,t)$ given in \eqref{ode} in the original space by 
\begin{equation}\label{ode_original}
   \frac{\partial u_a (\bsx, t)}{\partial t}
    =
    \frac{ \tsc \imagunit}{2} \, \nabla^2 u_a(\bsx, t) - \frac{\imagunit}{\tsc} \mathcal{I}_{\bsn}(v \, u_a)(\bsx,t)
    .
\end{equation}
Here we see two different dynamics in \eqref{eq:tdse_intp} and \eqref{ode_original}, therefore, by letting $\theta_\bsn (\bsx,t):=\mathcal{I}_{\bsn} ( u) (\bsx, t) -  u_a(\bsx, t)$ and comparing \eqref{eq:tdse_intp} with \eqref{ode_original}, we have
\begin{equation}\label{eq:theta_}
 \frac{\partial  \theta (\bsx, t)}{\partial t}
     =
    \frac{ \tsc \imagunit}{2} \,  \nabla^2 \theta (\bsx, t) - \frac{\imagunit}{\tsc} \mathcal{I}_{\bsn} ( v \, \theta )(\bsx,t) +\delta_\bsn (\bsx, t).
\end{equation}
We note that $\theta(\bsx,0)=0$. Using the relation 
\[\frac{1}{2}\frac{\partial  |\theta (\bsx, t)|^2}{\partial t}
=
\operatorname{Re}\left( \theta(\bsx,t) \overline{\frac{\partial \theta (\bsx, t)}{\partial t}} \right),
\]
where $\overline{x}$ denotes the complex conjugate, and using the chain rule we obtain the following inequality
\begin{align*}
 \|\theta(\cdot,t)\|_{L_2} \frac{\partial  \| \theta (\cdot, t) \|_{L_2} }{\partial t} 
 &=
 \frac{1}{2} \frac{\partial  \| \theta (\cdot, t) \|_{L_2}^2 }{\partial t}  
 =
 \operatorname{Re} \left( \biggl \langle \theta (\cdot, t), \frac{\partial \theta (\cdot,t)}{\partial t}  \biggr\rangle_{L_2} \right) \\
 &\kern-7em =
 \operatorname{Re} \left( \biggl \langle \theta (\cdot, t),   \frac{ \tsc \imagunit}{2} \,  \nabla^2 \theta (\cdot, t) - \frac{\imagunit}{\tsc} \mathcal{I}_{\bsn} ( v \, \theta )(\cdot,t)   \biggr\rangle_{L_2} \right) 
 + \operatorname{Re} \left( \biggl \langle \theta (\cdot, t),   \delta_\bsn (\cdot, t)  \biggr\rangle_{L_2} \right) \\
 &= 
 \operatorname{Re} \left( \biggl \langle \theta (\cdot, t),    \delta_\bsn (\cdot, t)  \biggr\rangle_{L_2} \right) 
 \le   
 \|\theta(\cdot,t)\|_{L_2}   \|\delta_{\bsn}(\cdot,t)\|_{L_2},
\end{align*}
where we used the fact that our discrete Fourier matrix $F_{\bsn}$ is unitary which makes the operator $ \frac{\tsc \imagunit}{2} \nabla^2 (\cdot) - \frac{\imagunit}{\tsc} \mathcal{I}_{\bsn} ( v \, \; (\cdot) )(\bsx,t)$ (e.g., \eqref{ode} and \eqref{ode_original}) self-adjoint,
and consequently the energy $ \biggl \langle \theta (\cdot, t),   \frac{ \tsc }{2} \,  \nabla^2 \theta (\cdot, t) - \frac{1}{\tsc} \mathcal{I}_{\bsn} ( v \, \theta )(\cdot,t)   \biggr\rangle_{L_2}$ is always real. 
Dividing both side of the above inequality by $\|\theta(\cdot,t)\|_{L_2}$ and integrating over time, we obtain
\begin{equation*}
\int_0^t \frac{\partial  \| \theta (\cdot, t') \|_{L_2}}{\partial t} \rd t' =\|\theta (\cdot,t)\|_{L_2} \le  \int _0 ^t  \| \delta_\bsn  (\cdot,t' ) \|_{L_2} \rd t'.
\end{equation*}
Using Lemma~\ref{lem:aliasing}, we can explicitly calculate the defect
\begin{align*}
\delta_{\bsn} (\bsx, t) &= \frac{ \tsc \imagunit}{2} \, \mathcal{I}_{\bsn} (\nabla^2 u(\bsx, t)) - \frac{ \tsc \imagunit}{2} \, ( \nabla^2  \mathcal{I}_{\bsn} (u) (\bsx, t) ) \\
&=
 \frac{ \tsc \imagunit}{2} \,  \sum_{\bsh \in \calA({\bsZ,\bsn})} \left(\sum_{\bsell\in  \Lambda^\bot(\bsZ,\bsn)} \kern-1em(\|\bsh+\bsell  \|_2 ^2 - \|\bsh\|_2^2) \; \widehat{u}(\bsh+\bsell,t)\right) \exp(\twopii\bsh\cdot\bsx)  .  
\end{align*}
For $\bsell=\bszero$, all terms become zero and we drop those. Now we use \eqref{eq:minimal_r} such that $| \; \|\bsh+\bsell\|_2^2-\|\bsh\|_2^2 \;| \le \|\bsh+\bsell\|_2^2$ for any $\bsh \in \calA(\bsZ,\bsn)$ and $\bsell\in  \Lambda^\bot(\bsZ,\bsn)$. This means
\begin{align*}
\| \delta_{\bsn} (\cdot, t) \|_{L_2} 
&\le
\frac{ \tsc}{2}  \left( \sum_{\bsh \in \calA({\bsZ,\bsn})} \left(\sum_{\bszero \ne \bsell\in  \Lambda^\bot(\bsZ,\bsn)} \kern-1em \|\bsh+\bsell  \|_2 ^2 \;|\widehat{u}(\bsh+\bsell,t)| \right)^2  \right)^{1/2} \\
&\le
\frac{ \tsc}{2}   \displaystyle\sum_{\bsh \in \Z^d\setminus \calA(\bsZ,\bsn)} \kern-1em \|\bsh \|_2 ^2 \; | \widehat{u}(\bsh,t)| .
\end{align*}
Therefore, we have 
\[
 \|\theta (\cdot,t)\|_{L_2} \le \int _0 ^t  \| \delta_\bsn  (\cdot,t' ) \|_{L_2} \rd t' 
 \le
 t \max_{0 \le t'\le t} \frac{ \tsc}{2}   \sum_{\bsh \in \Z^d\setminus \calA(\bsZ,\bsn)} \kern-1em \|\bsh \|_2 ^2 \; | \widehat{u}(\bsh,t')| .
\]
For the remaining term $\|u(\cdot,t)-\mathcal{I}_{\bsn} ( u) (\cdot, t)\|_{L_2}$, we have
\begin{align*}
  \|u(\cdot,t)-\mathcal{I}_{\bsn} (u) (\cdot, t)\|_{L_2} \nonumber \\
  &\kern-4em=
  \left({\displaystyle\sum_{\bsh \in \Z^d\setminus \calA(\bsz,n)}|\widehat{u}(\bsh,t)|^2} + {\displaystyle\sum_{\bsh \in \calA(\bsz,n)} \left| \sum_{\bszero \ne \bsell\in  \Lambda^\bot(\bsZ,\bsn)} \kern-1em \widehat{u}(\bsh+\bsell,t) \right|^2}\right)^{1/2}\\
   &\kern-4em\le
  \left({\displaystyle\sum_{\bsh \in \Z^d\setminus \calA(\bsz,n)}\kern-1em  | \widehat{u}(\bsh,t)|^2}\right)^{1/2} + \left({\displaystyle\sum_{\bsh \in \calA(\bsz,n)} \left| \sum_{\bszero \ne \bsell\in  \Lambda^\bot(\bsZ,\bsn)} \kern-1em \widehat{u}(\bsh+\bsell,t) \right|^2}\right)^{1/2}\\
     &\kern-4em\le
  {\displaystyle\sum_{\bsh \in \Z^d\setminus \calA(\bsz,n)}\kern-1em  | \widehat{u}(\bsh,t)|} + \left({\displaystyle\sum_{\bsh \in \calA(\bsz,n)} \left( \sum_{\bszero \ne \bsell\in  \Lambda^\bot(\bsZ,\bsn)} \kern-1em |\widehat{u}(\bsh+\bsell,t)| \right)^2}\right)^{1/2}\\
  &\kern-4em\le
  2{\displaystyle\sum_{\bsh \in \Z^d\setminus \calA(\bsz,n)}\kern-1em  | \widehat{u}(\bsh,t)|}.
  \end{align*}

Using the triangle inequality, we obtain
\begin{align*}
\|u(\cdot,t) - u_a(\cdot,t) \|_{L_2} 
 &\le
\|u(\cdot,t)-\mathcal{I}_{\bsn} ( u) (\cdot, t)\|_{L_2} + \|\mathcal{I}_{\bsn} ( u) (\cdot, t)- u_a(\cdot,t) \|_{L_2} \\
 &\le
 2 \sum_{\bsh \in \Z^d\setminus \calA(\bsZ,\bsn)}| \widehat{u}(\bsh,t) | + \frac{t\tsc}{2}\max_{0 \le t'\le t}\sum_{\bsh \in \Z^d\setminus \calA(\bsZ,\bsn)} \|\bsh\|_2^2 \; | \widehat{u}(\bsh,t') |.
\end{align*}
This completes the proof.
\end{proof}
The above error bound is further bounded by
\[
\|u(\cdot,t) - u_a(\cdot,t) \|_{L_2} \le   (2 + \frac{t\tsc}{2} ) \max_{0 \le t'\le t}\sum_{\bsh \in \Z^d\setminus \calA(\bsZ,\bsn)} \|\bsh\|_2^2 \; | \widehat{u}(\bsh,t') |.
\]
This is similar to the result of \cite[Theorem 1.8]{MR2474331} for the one-dimensional case which states 
\[
\|u(\cdot,t) - u_a(\cdot,t) \|_{L_2} \le C(1+t) \max_{0 \le t'\le t}  \| \frac{\partial^2 u(\cdot,t')}{ \partial x^2}- (\frac{\partial^2  \mathcal{I}_{\bsn} (u) }{ \partial x^2}) (\cdot, t')\|_{L_2}.
\]
For certain function spaces, the approximation errors of lattice points are explicitly known, e.g., \cite{byrenheid2016tight,KSW2006}. It might be possible to construct approximation lattices according to the referenced papers and then to extend the frequency index set to fulfill the needed conditions. 
However, this is not the focus of the present paper. The focus is the interplay between the spatial discretization and the time-stepping error, because the time-stepping error itself is heavily affected by the spatial discretization as we can see from the comparison with \cite{G07}.

\section{Conclusion}\label{sec:5}

We approximated the solution of the time-dependent Schrödinger equation by using rank-$1$ and rank-$r$ lattices for the space discretization and Strang splitting for the time discretization. We combined the anti-aliasing set of the lattices together with FFTs to obtain both theoretical advantages and computational efficiency.
We showed that the time discretization of our method has second-order convergence for a potential function $v \in E_\alpha(\T^d)$ with $\alpha > 9/2$ which is independent of the dimension $d$. 
The numerical experiments confirm the theory.
We observed second order convergence with respect to the time step in cases up to $12$~dimensions.
Previous results based on sparse grids \cite{G07} have difficulty for cases higher than $5$~dimensions.

Here we also remark limitations of our method. We exploited the structure of lattices to mitigate the curse of dimensionality, but we do not completely remove the curse. This means, we can solve rather higher-dimensional problems than regular grids and sparse grids in \cite{G07} can, but not too high.
Also, our focus of the present paper is on the time-dependent problems. The algorithm is especially made for obtaining a small time-stepping error. Therefore, we cannot expect that our method works better for the time-independent problems than existing methods such as \cite{MR3882754, MR2339626}, for this the lattice points have to be constructed with this in mind.

Our method can be applied to different problems which would be more interesting for physics applications. One possibility is the time-dependent non-linear Schr{\"o}dinger equation for simulating Bose--Einstein condensates. 
In \cite{MR3022261}, Thalhammer showed that pseudo-spectral Fourier methods using regular grids with exponential splitting can obtain the higher order convergence in time stepping. We may possibly alternate the regular grid with lattice points to obtain the efficient simulation scheme with keeping the same convergence order.
Another possibility is using our method for time-dependent potentials. For instance, the time-dependent harmonic oscillator is used for considering multiphoton excitation of molecules, see \cite{kosloff1983fourier}.
Our method can also be extended to the higher-order exponential splitting, which is studied in the following up paper \cite{2019arXiv190506904S}.

\section*{Acknowledgments}
We would like to thank two anonymous referees for their valuable comments. We also thank financial supports from the KU Leuven research fund.
We thank Christian Lubich for his valuable comments on the proof of Lemma~\ref{lem:u-series}.

\bibliographystyle{siamplain}
\bibliography{schr}


\end{document}